\documentclass{article}
\usepackage{amsmath, amssymb, amsthm, amsfonts, amsxtra, color, graphicx, hyperref, mathtools, tikz, tikz-cd, url, soul}
\usepackage{thmtools}
\usepackage{thm-restate}
\usepackage{cleveref, appendix}

\input xy
\xyoption{all}

\numberwithin{equation}{section}

\declaretheorem[name=Theorem]{thm}
\theoremstyle{plain}

\newtheorem{lemma}{Lemma}
\newtheorem{proposition}{Proposition}
\newtheorem{conjecture}{Conjecture}
\theoremstyle{definition}
\newtheorem{definition}{Definition}

\newtheorem{question}{Question}

\def\calF{\mathcal{F}}
\def\EE{\mathbb{E}}

\renewcommand{\a}{\alpha}

\newcommand{\lmax}{\lambda_{\textup{max}}}
\newcommand{\lmin}{\lambda_{\textup{min}}}

\begin{document}

\title{\vspace*{-1.1in}Expected Chromatic Number of Random Subgraphs}
\author{Ross Berkowitz\\ \texttt{ross.berkowitz@yale.edu} \and Pat Devlin\\ \texttt{patrick.devlin@yale.edu} \and Catherine Lee\\ \texttt{catherine.lee@yale.edu} \and Henry Reichard\\ \texttt{henry.reichard@yale.edu} \and David Townley\\ \texttt{david.townley@yale.edu}\\ \\ Department of Mathematics\\Yale University, New Haven, CT}
\date{November 5, 2018}

\renewcommand{\thefootnote}{\fnsymbol{footnote}}


\maketitle

\begin{abstract}
Given a graph $G$ and $p \in [0,1]$, let $G_p$ denote the random subgraph of $G$ obtained by keeping each edge independently with probability $p$.  Alon, Krivelevich, and Sudokov \cite{alon} proved $\mathbb{E} [\chi(G_p)] \geq C_p \frac{\chi(G)}{\log |V(G)|}$, and Bukh \cite{bukh} conjectured an improvement of $\mathbb{E}[\chi(G_p)] \geq C_p \frac{\chi(G)}{\log \chi(G)}$.  We prove a new spectral lower bound on $\mathbb{E}[\chi(G_p)]$, as progress towards Bukh's conjecture.  We also propose the stronger conjecture that for any fixed $p \leq 1/2$, among all graphs of fixed chromatic number, $\mathbb{E}[\chi(G_p)]$ is minimized by the complete graph.  We prove this stronger conjecture when $G$ is planar or $\chi(G) < 4$.  We also consider weaker lower bounds on $\mathbb{E}[\chi(G_p)]$ proposed in a recent paper by Shinkar ~\cite{shinkar}; we answer two open questions posed in \cite{shinkar} negatively and propose a possible refinement of one of them.
\end{abstract}

\section{Introduction}
For a graph $G$ and $p \in [0,1]$, we obtain a probability distribution $G_p$ called a \emph{random subgraph} by taking subgraphs of $G$ with each edge appearing independently with probability $p$. When $G = K_n$ is the complete graph on $n$ vertices, this is called the \emph{Erd\H{o}s--R\'enyi random graph}, denoted $G(n,p)$.  A \emph{proper coloring} of a graph is an assignment of colors to the vertices such that no two adjacent vertices are the same color. Finally, the \emph{chromatic number} $\chi(G)$ of a graph is the minimal number of colors needed to construct a proper coloring.

The chromatic number is one of the most important parameters of a graph, and many problems in computer science---e.g., register allocation, pattern matching, and scheduling problems---can be reduced to finding the chromatic number of a given graph. In the probabilistic setting, the distribution of $\chi(G_p)$ is studied in statistical mechanics, where physicists use random subgraphs to model molecular interactions, and properties of the resulting graph colorings are predictive of various macroscopic features ~\cite{huck}.

The chromatic number of the Erd\H{o}s--R\'enyi graph has been particularly well studied, and (for $p$ constant) Bollob\'as \cite{bollobas} was able to show that $\mathbb{E}[\chi(G(n,p))] \sim c_p \cdot n/\log{n} = c_p \chi(K_n)/\log{\chi(K_n)}$, where $c_p$ is a constant depending on $p$, and the notation $A \sim B$ is used to mean that $A/B$ tends to $1$ as the relevant parameter (here $n$) tends to infinity.  For general graphs, one cannot hope for such tight control over $\chi(G_p)$ only in terms of $\chi(G)$. The trivial upper bound $\chi(G_p) \leq \chi(G)$ is asymptotically best possible when $G$ is a disjoint union of many cliques, and the first general lower bound was was given by Alon, Krivelevich, and Sudakov \cite{alon}, who proved $\chi(G_{1/2}) \geq \frac{\chi(G)}{2 \log_2 |V(G)|}$ almost surely (i.e., with probability tending to $1$).  However, this ceases to be a meaningful bound when $|V(G)| \gg \chi(G)$, and Bukh ~\cite{bukh} asks whether it can be improved by eliminating the dependence on $|V(G)|$.

\begin{question}[Bukh]\label{bukh question}
    For each $p \in [0,1]$, is there a constant $c_p>0$ such that $\mathbb{E}[\chi(G_p)] > c_p \cdot \frac{\chi(G)}{\log \chi(G)}$ for all graphs $G$?
\end{question}

%If true, the following conjecture would, in combination with Bollob\'as' result, establish an affirmative answer to Bukh's question for $p \leq 1/2$:
\noindent Thus, in light of Bollob\'as's result, Bukh asks whether $\mathbb{E}[\chi(G_p)]$ is (up to a constant) minimized by $G= K_n$. 

While this question is still unresolved, several papers have made progress towards an affirmative answer.  In addition to providing concentration results on $\chi(G_p)$, Shinkar \cite{shinkar} proved that if $|V(G)| = n$, then
\begin{equation}\label{shinkar equation}
\mathbb{E}[n/\alpha(G_p)] \geq C_p \frac{n/\alpha(G)}{\log{(n/\alpha(G))}},
\end{equation}
where $\alpha(H)$ denotes the independence number of $H$ (i.e., the maximum size of a set of vertices containing no edges), and Mohar ~\cite{mohar} proved an affirmative analog to Bukh's question when each instance of $\chi$ is replaced by the \emph{fractional chromatic number}, $\chi_f$.  Though incomparable, these results are related by the general fact that $|V(H)| / \alpha(H) \leq \chi_f (H) \leq \chi(H)$ for all $H$.  Thus, these affirmatively resolve question \ref{bukh question} for any graph for which $n/\alpha(G)$ (or somewhat more generally $\chi_f (G)$) is within a multiplicative factor of $\chi(G)$, which by \cite{bollobas} includes almost all graphs.  The only other general lower bound on $\chi(G_p)$ is a short coupling argument in \cite{shinkar} showing $\mathbb{E}[\chi(G_{1/m})] \geq \chi(G)^{1/m}$ for all positive integers $m$.

As one of the main results of this paper, we use recent developments from random matrix theory and a celebrated result of Hoffman \cite{hoff} to obtain a new spectral lower bound on $\EE[\chi(G_p)]$.  Recalling relevant definitions, for a graph $H$ its \emph{adjacency matrix} is the matrix indexed by $V(H)$ whose $(u,v)$-entry is $1$ if $u \sim v$ and $0$ otherwise.  Because this matrix is real-symmetric, all its eigenvalues are real, and we may define $\lmin(H)$ and $\lmax (H)$ to be its least and greatest eigenvalues (respectively).  We prove the following.

\begin{restatable}{thm}{spectral}\label{spectral}
There is a constant $C > 0$ such that for each $p \in (0,1)$, for any graph with maximum degree $\Delta$ and $n = |V(G)|$, we almost surely have
\[
    \frac{\lmax (G_p)}{-\lmin (G_p)} \geq \frac{\lmax (G) -(C/p) (\sqrt{\Delta} + \sqrt{\log(n)})}{- \lmin (G) + (C/p) (\sqrt{\Delta} + \sqrt{\log(n)})}.
\]
In particular, almost surely
$\chi(G_p) \geq \frac{\lmax (G)}{- \lmin (G) + (C/p) (\sqrt{\Delta} + \sqrt{\log(n)})}.$
\end{restatable}
\noindent The second part of the above follows from Hoffman's result that $n / \alpha(H) \geq 1 + \frac{\lmax (H)}{-\lmin (H)}$, which gives an affirmative answer to question \ref{bukh question} provided essentially that Hoffman's bound differs from $\chi(G)$ by at most a factor of $\log \chi(G)$ and that $-\lmin (G)$ is not much less than $\sqrt{\Delta} + \sqrt{\log(n)}$.  For instance, we show that there is an infinite family of graphs---namely appropriately chosen Kneser graphs---for which our spectral bound implies Bukh's bound, while none of the other general bounds on $\mathbb{E}[\chi(G_p)]$ are able to.  (Although for Kneser graphs, the behavior of $\alpha(G_p)$ and $\chi(G_p)$ is already well-understood \cite{alphaKneser4, chiKneser}.)%\cite{alphaKneser1, alphaKneser2, alphaKneser3, alphaKneser4} and also \cite{chiKneser}.)

In addition to this spectral bound, we also propose and study the following two conjectures, which are readily seen as weaker and stronger (resp.) than an affirmative answer to question \ref{bukh question}.

\begin{conjecture}\label{power conjecture}
For each $p \in [0,1]$, there is a constant $c_p$ such that $\EE[\chi(G_p)] \geq c_p \cdot \chi(G)^p$ for all $G$.
\end{conjecture}

\begin{conjecture}\label{minimize}
For each $p \leq \frac{1}{2}$, we have $\EE[\chi(G_p)] \geq \EE[\chi(G(n,p))]$ for all $G$ with $\chi(G) = n$.
\end{conjecture}

\noindent Conjecture \ref{power conjecture} was originally posed in \cite{shinkar} with the constant $c_p = 1$; however, we show that in general $c_p < 1$ is in fact required.  As discussed in section \ref{section power bound}, in the case $p=1/m$, a proof of this conjecture (with $c_p = 1$) follows from the classic result of Zykov \cite{ore} that $\chi(G \cup H) \leq \chi(G) \chi(H)$.  However, the tightness in our following generalization of this result highlights a barrier to this approach for general $p$.

\begin{restatable}{thm}{zykovGeneralization}\label{zykov generalization}
Fix integers $0 \leq t < n$.  Let $G$ be a graph and $G_1, G_2, \dots , G_n \subset G$ such that every edge of $G$ lies in at least $n-t$ of the $G_i$.  Then $\chi(G_1)\chi(G_2)\cdots\chi(G_n) \geq \chi(G)^{n/(t+1)}$. Furthermore, for any $t < n$, there are examples with $\chi(G)$ arbitrarily large for which this bound is tight.
\end{restatable}

As for Conjecture \ref{minimize}, we first prove that a condition such as $p \leq \frac{1}{2}$ is necessary in the following sense.
\begin{restatable}{thm}{ungeneralizable} \label{thm:ungeneralizable}
For any graph $H$, if $1 - \frac{1}{|E(H)|} < p < 1$, there exists a $G$ with $\chi(G) = \chi(H)$ such that $\EE[\chi(H_p)] > \EE[\chi(G_p)]$.
\end{restatable}

\noindent However, we are in fact able to prove the following special case of Conjecture \ref{minimize}.

\begin{restatable}{thm}{planar-bukh} \label{planar-bukh}
Suppose $\chi(G) = n$ and $p \leq \frac{1}{2}$.  If $G$ is planar or if $n < 4$, then $\EE[\chi(G_p)] \geq \EE[\chi(G(n,p))]$.
\end{restatable}

\noindent We also exhibit numerical evidence supporting Conjecture \ref{minimize} for Mycielski graphs.  As Mycielski graphs are prototypical examples of triangle-free graphs with high chromatic number, they are perhaps the most natural candidates for a possible counterexample to our conjecture.  Our numerical exploration of these graphs also suggests some very interesting structure in the distribution of $\chi(G_p)$, which we feel is of sufficient independent interest to warrant its own study.

\subsection{Outline of our paper}
We begin with section \ref{section spectral}, in which we prove our spectral result of Theorem \ref{spectral}.  We continue in section \ref{section power bound} with a discussion of conjecture \ref{power conjecture} and proof Theorem \ref{zykov generalization}.  In section \ref{section mycielskians}, we present Mycielskian graphs in the context of Conjecture \ref{minimize} and use them to prove Theorem \ref{thm:ungeneralizable}.  Section \ref{section proof of planar-bukh} is devoted to a proof of Theorem \ref{planar-bukh} with some of the casework placed in an appendix.  Finally, in section \ref{section kneser} we show how our spectral bound can be applied to Kneser graphs, and we state and disprove a related conjecture of Shinkar on the chromatic number of induced subgraphs.

\vspace*{12pt}
\noindent \textbf{Acknowledgement:} We would like to thank the support and funding of the 2018 \textit{Summer Undergraduate Math Research at Yale} (SUMRY) program, at which this project was completed.

\section{Spectral bound: proof of Theorem \ref{spectral}}\label{section spectral}
\noindent Among the spectral bounds on the chromatic number, the first (and best-known) is due to Hoffman \cite{hoff}:
\begin{align*}
\chi(G) \geq 1 + \frac{\lmax (G)}{- \lmin (G)} \ ,
\end{align*}
where $\lmax$ and $\lmin$ are the maximum and minimum eigenvalues of $G$'s adjacency matrix, $A_G$.  In order to use Hoffman's bound to obtain a lower bound on the expected chromatic number, we need to estimate the variability in the eigenvalues of $A_{G_p}$.  For this, we appeal to a result of Bandeira and van Handel, which appears as a special case of Corollary 3.12 (see also remark 3.13) of ~\cite{bandeira}.  Here, we cite only a special case suited for our needs.
\begin{thm}[Bandeira and van Handel]\label{thm bandeira}
Let $X$ be an $n \times n$ symmetric matrix whose entries are independent\footnote{That is to say $X_{i,j}$ is independent of every other entry except $X_{j,i}$.} mean $0$ random variables of magnitude at most 1.  There is a universal constant $C$ such that
\[
\mathbb{P}\left( \Vert X \Vert \geq C \left( \sigma + \sqrt{\log(n)} \right) \right) \leq n^{-100},
\]
where $\Vert X \Vert = \displaystyle \sup_{\vec{0} \neq \vec{u} \in \mathbb{R}^n} \frac{\Vert X \vec{u} \Vert _2}{\Vert \vec{u} \Vert_2}$ is the operator norm of $X$, and $\sigma = \displaystyle \max_{i} \sqrt{ \sum_{j} \mathbb{E}[X_{i,j} ^2]}$.
\end{thm}

\noindent With this, we can prove our lower bound on the spectrum of $G_p$.

\begin{proof}[Proof of Theorem \ref{spectral}]
We wish to relate the eigenvalues of $A_{G_p}$ with those of $A_G$.  For this, consider the random $n \times n$ matrix $X = p A_{G_p} - A_G$.  Since the eigenvalues of $p A_G$ are just $p$ times those of $A_G$, we will be able to control the eigenvalues of $A_{G_p}$ provided that $\Vert X \Vert$ is small.  Then $X$ is symmetric with independent entries of mean $0$, which are each bounded in absolute value by $1$.  Thus, $X$ satisfies the conditions of Theorem \ref{thm bandeira} with $\displaystyle \sigma = \max_{i} \sqrt{\sum_{j} \mathbb{E}[X^{2} _{i,j}]} \leq \sqrt{\Delta},$ implying
\[
\mathbb{P}\left( \Vert X \Vert \geq C \left( \sqrt{\Delta} + \sqrt{\log(n)} \right) \right) \leq n^{-100}.
\]

\noindent For any $n \times n$ matrix $M$ and any $\vec{0} \neq \vec{u} \in \mathbb{R}^n$ consider the \textit{Rayleigh quotient,} $R(M, \vec{u}) = \dfrac{\langle \vec{u}, M \vec{u} \rangle}{\langle \vec{u}, \vec{u} \rangle}.$  For symmetric matrices, it is well known that $\lmax (M) = \sup_{\Vert \vec{u} \Vert =1} R(M, \vec{u})$ and $\lmin (M) = \inf_{\Vert \vec{u} \Vert =1} R(M, \vec{u})$.  Thus, for any symmetric matrices $M$ and $N$ we have
\begin{eqnarray*}
\lmax(M) &=& \sup_{\Vert \vec{u} \Vert = 1} R(M, \vec{u}) = \sup_{\Vert \vec{u} \Vert= 1} \Big[ R(N, \vec{u})+ R(M-N, \vec{u}) \Big] \leq \sup_{\Vert \vec{u} \Vert = 1} \Big[ R(N, \vec{u}) \Big] + \sup_{\Vert \vec{u} \Vert = 1}\Big[ R(M-N, \vec{u}) \Big]\\
&=& \lmax(N) + \sup_{\Vert \vec{u} \Vert = 1} \langle \vec{u}, (M-N) \vec{u} \rangle \leq \lmax(N) + \Vert M-N \Vert,
\end{eqnarray*}
where the last inequality comes from the definition of the operator norm and the Cauchy-Schwarz inequality.  Thus, $|\lmax(M) - \lmax(N)| \leq \Vert M-N \Vert$ and by similar reasoning, $|\lmin(M) - \lmin(N)| \leq \Vert M-N \Vert$.

From this, we see
\begin{eqnarray*}
\Vert X \Vert &\geq& |\lmax(pA_{G}) - \lmax(A_{G_p})| = |p\lmax(A_{G}) - \lmax(A_{G_p})|, \qquad \text{and}\\
\Vert X \Vert &\geq& |\lmin(pA_{G}) - \lmin(A_{G_p})| = |p\lmin(A_{G}) - \lmin(A_{G_p})|.
\end{eqnarray*}

\noindent And since almost surely $\Vert X \Vert \leq C \left(\sqrt{\Delta} + \sqrt{\log(n)} \right)$, a simple rearrangement completes the proof.
\end{proof}

After combining this with Hoffman's bound, we almost surely have the lower bound
\[
\chi(G_p) \geq \frac{\lmax (G)}{- \lmin (G) + (C/p) (\sqrt{\Delta} + \sqrt{\log(n)})}.
\]
Note that if $\Delta > \log(n)$, we could absorb the $\sqrt{\log{n}}$ term into the constant, and since $\lmax(G) \geq 2|E(G)| / n$, we almost surely have the more compact
\[
\chi(G_p) \geq \frac{\lmax (G)}{- \lmin (G) + (C/p)\sqrt{\Delta}} \geq \frac{2|E(G)| / n}{- \lmin (G) + (C/p)\sqrt{\Delta}}, \qquad \qquad \text{provided that $\Delta > \log(n)$.}
\]

\section{Discussion of Conjecture \ref{power conjecture}}\label{section power bound}
\noindent Let us now turn our attention to Conjecture \ref{power conjecture}.  As a warm-up (and helpful example), suppose that $G$ is an odd cycle on $2k+1$ vertices.  Then we have $\mathbb{E}[\chi(G_{p})] = 2 + p^{2k+1} -(1-p)^{2n+1}$.  Therefore, for $p \in (0,1)$ we have $\lim_{k \to \infty} \EE[\chi(G)_{p}] = 2$.  On the other hand, $\chi(G) = 3$, which shows $\EE[\chi(G_p)] / \chi(G) ^{p} \to 2 \cdot 3^{-p}$.  Thus, if Conjecture \ref{power conjecture} holds, we need $c_p \leq 2 \cdot 3^{-p}$, which is already less than 1 when $p = 2/3$.

On the other hand, consider the following proof of Conjecture \ref{power conjecture} when $p = 1/m$ for positive integer $m$.  We first randomly assign each edge of $G$ to an element of $\{1,2, \ldots, m\}$, and let $G^{i}$ denote the edges labelled $i$.  Clearly $\chi(G^1 \cup G^2 \cup \cdots \cup G^m) = \chi(G)$, so we have
\[
\chi(G) ^{1/m} = \chi(G^1 \cup G^2 \cup \cdots \cup G^m) ^{1/m} \leq \left(\prod_{i=1} ^{m} \chi(G^{i}) \right) ^{1/m} \leq \dfrac{1}{m} \sum_{i=1} ^{m} \chi(G^{i}),
\]
where the last inequality holds by the AM-GM inequality.  Taking the expected value of both sides and using the fact that each $G^{i}$ has the same distribution as $G_{1/m}$, we obtain
\[
\chi(G) ^{1/m} \leq \EE \left[ \dfrac{1}{m} \sum_{i=1} ^{m} \chi(G^{i}) \right] = \dfrac{1}{m} \sum_{i=1} ^{m} \EE [\chi(G^{i})] = \EE[\chi(G_{1/m})].
\]

Now suppose for motivation that we would like to prove something like $\mathbb{E}[\chi(G_{2/3})] \geq c_{2/3} \ \chi(G)^{2/3}$ in a similar way.  We could consider a construction as above to get a partition of $G$ into $3$ disjoint graphs $G^1, G^2, G^3$ and then consider the graphs $G^{I} = \cup_{i \in I} G^{i}$ with $|I|=2$.  With this, each $G^{I}$ has the same distribution as $G_{|I|/3}$, and we know that each edge of $G$ appears in $2$ elements of $\{G^{I} \ : \ |I|=2\}$.  Proceeding as before, we would hope for a bound such as $\chi(G^{1,2}) \chi(G^{1,3}) \chi(G^{2,3}) \geq \chi(G)^{2}$, and in fact replacing the exponent on the right-hand-side with anything greater than $3/2$ would improve on the trivial bound $\EE[ \chi(G_{2/3}) ] \geq \EE[\chi(G_{1/2})] \geq \chi(G)^{1/2}$.  However this is not possible in general, and arguments that only use that each edge shows up in the correct number of $G^{I}$ cannot improve on these trivial bounds.

\zykovGeneralization*

\begin{proof}
To obtain the lower bound, observe that since each edge in $G$ lies in at least $n-t$ of the $G_i$, all of $G$'s edges must be contained in the union of any $t+1$ distinct $G_i$. Hence, for any $\{i_1,\dots,i_{t+1} \} \subset [n]$ we have $\chi(G_{i_1})\cdots \chi(G_{i_{t+1}}) \geq \chi(G_{i_1}\cup \dots \cup G_{i_{t+1}}) = \chi(G)$, which yields the desired result by taking the product over all $(t+1)$-element subsets of indices.

To construct a family of examples for which our result is tight, let $q > n$ be any prime, let $\mathbb{F}_q$ denote the field with $q$ elements, and let $V = \{f(x) \in \mathbb{F}_q [x] : \deg(f) \leq t\}$ denote the set of polynomials over $\mathbb{F}_q$ of degree at most $t$.  Because $n > t$, two polynomials in $V$ are equal as functions iff all their coefficients are equal, and $|V| = q^{t+1}$.  Let $G$ be the complete graph on $V$, and for each $i \in \{1, 2, \ldots , n\}$, let $G_i$ be the graph on $V$ with $f \sim g$ iff $f(0) + f(i) \neq g(0) + g(i)$.  Then $G_i$ is a complete $q$-partite graph where $\{f \in V : f(0) + f(i) = c\}$ is independent for each $c \in \mathbb{F}_q$.  So $\chi(G_i) = q$ for all $i \leq n$, implying $\chi(G_1) \chi(G_2) \cdots \chi(G_n) = q^n = \chi(G)^{n/(t+1)}$.

We claim that in this construction, each edge appears in at least $n-t$ graphs $G_i$.  To see this, suppose the edge $f\sim g$ is missing from $G_i$ for all $i$ in some set $I$.  This implies $f(i) + f(0) - g(i) - g(0) = 0$ for all $i \in I$.  But $f(x) + f(0) - g(x) - g(0)$ is a polynomial of degree at most $t$, which is equal to $0$ for all $i \in I$.  Thus, either $|I| \leq t$ or else we need $f - g$ is the zero polynomial, implying $f = g$.
\end{proof}

\section{Need for $p<1- \varepsilon$ in Conjecture \ref{minimize}}\label{section mycielskians}
Let us recall the following well-known construction of Mycielski \cite{mycielski}.  For a graph $G$ on vertex set $V$, let $M(G)$ denote the graph with vertex set $V \times \{0,1\} \cup \{x\}$ and with the edges $(v,1) \sim x$ for all $v \in V$ as well as all the edges of the form $(u,i) \sim (v,j)$ where $u \sim_{G} v$ and $i \neq j$.

%It is immediately clear that $|V(M(G))| = 2|V(G)| + 1$ and $|E(M(G))| = 2|E(G)| + |V(G)|$, and it is not difficult to see that $\chi(M(G)) = \chi(G) + 1$ (see \cite{mycielski}).

With this, we can define the sequence of \textit{Mycielskian graphs} where $M_2$ is the 2-vertex graph with a single edge, and $M_k = M(M_{k-1})$ for all $k \geq 3$.  To get a feeling for Conjecture \ref{minimize}, consider Figure 1, which contains plots of $\mathbb{E}[\chi(G(k,p)]$ and $\mathbb{E}[\chi((M_k)_p)]$ (viewed as functions of $p$) for several small values of $k$.

\begin{figure}[h]
    \centering
    \begin{minipage}{0.45\textwidth}
        \centering
 \includegraphics[width=3in]{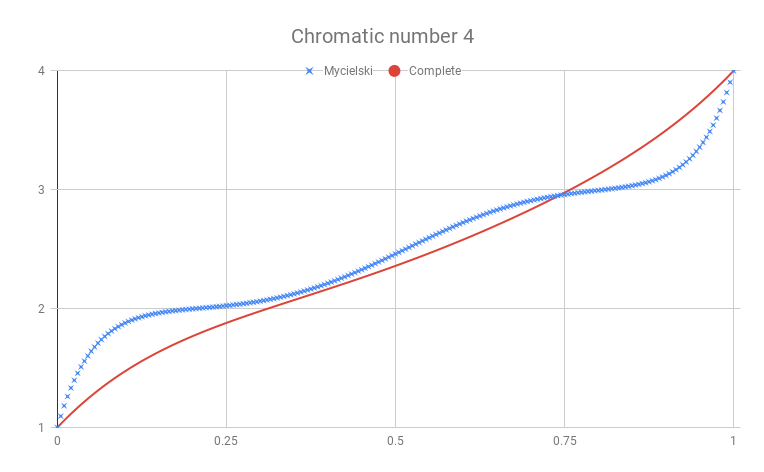}
 % first figure itself
 %       \caption{first figure}
    \end{minipage}\hfill
    \begin{minipage}{0.45\textwidth}
        \centering
\includegraphics[width=3in]{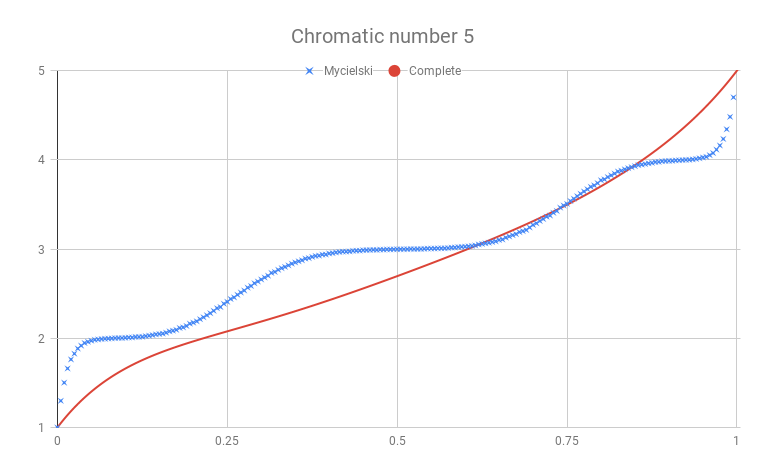}
 % second figure itself
%        \caption{second figure}
    \end{minipage}
\includegraphics[width=3.5in]{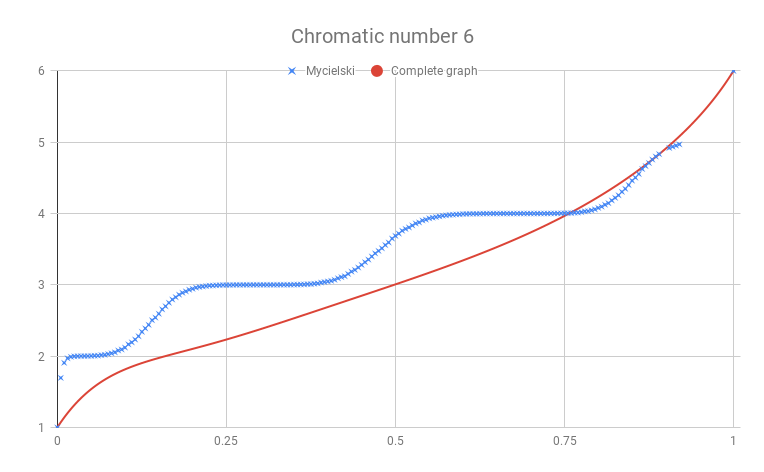}    
\caption{Plots of $\mathbb{E}[\chi(G_p)]$ for $M_k$ (thicker, blue) and $K_k$ (thinner, red) for $k=4, 5,6$}
\end{figure}

For every $G$, it is not difficult to see that $\chi(M(G)) = \chi(G) + 1$.  Thus, since $\chi(M_k) = k$, Conjecture \ref{minimize} asserts that $\mathbb{E}[\chi((M_k)_p)] \geq \mathbb{E}[\chi(G(k,p))]$ whenever $p \leq \frac{1}{2}$, which---from Figure 1---we see to be true for $4 \leq k \leq 6$.  Although these plots agree with Conjecture \ref{minimize} for $p \leq \frac{1}{2}$, we also see that each has values of $p$ near $1$ for which the inequality of our conjecture fails (because $M_k$ has more edges than the complete graph on $k$ vertices).  In fact, we show that this is unavoidable in the following sense.

\ungeneralizable*

\begin{proof}An \textit{edge-critical graph} is one in which every proper subgraph has lower chromatic number.  For each $n \geq 3$, there are graphs $G$ with arbitrarily many edges and fixed $\chi(G) = n$---for instance we could obtain such a graph by iterating the Mycielskian construction starting with a large odd cycle.\footnote{It is easy to see that the Mycielskian construction preserves edge-criticality and that odd cycles are edge critical.}  Thus, we can select an edge-critical $G$ such that $1-(1-p) |E(H)| > p^{|E(G)|}$ and $\chi(G) = \chi(H) = n$.  For this $G$, edge-criticality implies $\mathbb{E}[\chi(G_p)] \leq n p^{|E(G)|} + (n-1) (1- p^{|E(G)|}) = n-1 + p^{|E(G)|} < n - (1-p) |E(H)|.$  On the other hand, for any graph $\chi(H) - \mathbb{E}[\chi(H_p)] \leq (1-p) |E(H)|$ since $(1-p) |E(H)|$ is equal to the expected number of edges removed going from $H$ to $H_p$, and removing an edge lowers the chromatic number by at most 1.  Thus we have $\mathbb{E}[\chi(G_p)] < n - (1-p)|E(H)| \leq \mathbb{E}[\chi(H_p)]$, as desired.
\end{proof}

As an aside, it is interesting to note the apparent ``plateaus" in the graphs of $\mathbb{E}[\chi((M_k)_p)]$.  For values of $p$ in these plateaus, it seems reasonable to conjecture that the distribution of $\chi((M_k)_p)$ is tightly concentrated on an integer value, and it would be interesting to study these graphs for large $k$.

\section{Proof of Theorem \ref{planar-bukh}}\label{section proof of planar-bukh}
Conjecture \ref{minimize} naturally leads to the following definition.

\begin{definition}
For a family of graphs $\calF$ and fixed $p\in (0,1)$, we say that $G \in \calF$ is an \textbf{$n$-minimizer} among $\calF$ if $\mathbb{E}[\chi(H_p)] \geq \mathbb{E}[\chi(G_p)]$ for all $H \in \calF$ with $\chi(H) = \chi(G) = n$.
\end{definition}

\noindent In the language of $n$-minimizers, Conjecture \ref{minimize} states that for all $p \leq \frac{1}{2}$, and all $n$, $K_n$ is an $n$-minimizer among all graphs.  And for $n \geq 3$, Theorem \ref{thm:ungeneralizable} states that no graph is an $n$-minimizer for all $p \in (1-\varepsilon, 1]$.

For small chromatic numbers, Conjecture \ref{minimize} is easy to verify, and the case $n \in \{1, 2\}$ there is nothing to show.  As the first interesting case, the classification of 3-minimizers is given by the following lemma.

\begin{proposition}\label{3-minimizing}
For each $p\in (0,\frac{1}{2})$, $K_3$ is the unique 3-minimizer; for $p=\frac{1}{2}$, every odd cycle is a 3-minimizer; and for each $p\in(\frac{1}{2},1)$, there are no 3-minimizers.
\end{proposition}

For this, we first need the following easy lemma.

\begin{lemma}\label{subgraph lemma}
Let $G$ be a graph and $H$ a proper subgraph.  Then $\EE[\chi(H_p)]<\EE[\chi(G_p)]$ for all $p \in (0,1)$.
\end{lemma}

\begin{proof}[Proof of Lemma \ref{subgraph lemma}]
For this, we couple $H_p$ and $G_p$ by first sampling the edges of $H$ and then sampling the remaining edges of $G$.  In this coupling we have $H_p \subseteq G_p$ implying $\chi(H_p) \leq \chi(G_p)$.  Moreover, strict inequality is possible (e.g., if $H_p$ does not have any edges but $G_p$ does).
\end{proof}

\begin{proof}[Proof of Proposition \ref{3-minimizing}]
Since every graph with chromatic number at least $3$ contains an odd cycle, we need only consider odd cycles in determining which graphs are $3$-minimizers.  Letting $C^{2k+1}$ denote the odd cycle on $2k+1$ vertices, we have
\[
\mathbb{E}[\chi(C^{2k+1} _p)] = 2 + p^{2k+1} - (1-p)^{2k+1}.
\]
For $0 < p < \frac{1}{2}$, this is minimized when $k = 3$.  When $p = 1/2$, this quanity is $2$ independent of $k$.  And for $p \in (\frac{1}{2}, 1)$, this quantity converges to $2$ from above as $k \to \infty$.
\end{proof}

In light of this, (and the four-color theorem for planar graphs) to finish the proof of Theorem \ref{planar-bukh}, we need only prove that for $p \leq \frac{1}{2}$, $K_4$ is the unique $4$-minimizer among all planar graphs.

\begin{proposition}\label{k4 planar minimizer}
For all $p \in (0, \frac{1}{2}]$, $K_4$ is the unique 4-minimizer among planar graphs.
\end{proposition}
\begin{proof}[Proof sketch]
Our proof relies on some rather involved case analysis, which we move to an appendix for ease of reading.  Here, we provide a very high-level proof sketch.

Our starting point is the Gr\"unbaum--Aksionov theorem that every planar graph with at most three 3-cycles is 3-colorable \cite{grunbaum, aksionov}.  From this, we construct a finite list of graphs that must be contained in any planar graph with chromatic number $4$ somewhat simplifying along the way for our purposes.  After this, we simply compare $K_4$ to this finite list of subgraphs and note that the expected chromatic number of $K_4$ is the greatest.  Full details available in the appendix.
\end{proof}

\section{Discussion of Theorem \ref{spectral} and a question of \cite{shinkar}}\label{section kneser}
\noindent Although the Hoffman bound is often a poor estimate for $\chi(G)$, there are nonetheless natural families of graphs for which our spectral result is the only known general result providing the bound of question \ref{bukh question}.  For example, we will present the \textit{Kneser graphs}, whose parameters are chosen so that none of the previously known bounds discussed in the introduction establishes Bukh's conjecture, yet Theorem \ref{spectral} does.

The \emph{Kneser graph} with parameters $n \geq k \geq 0$, denoted $KG_{n,k}$, is the graph whose vertices are indexed by the $k$-element subsets of $\{1, 2, \ldots , n\}$ and for which two vertices are adjacent iff the corresponding sets are disjoint.  In this language, the classic Erd\H{o}s--Ko--Rado theorem \cite{erdos} states for $n \geq 2k$, $\alpha(KG_{n,k}) = {n-1 \choose {k-1}}$, and a celebrated result of Lov\'asz \cite{lovasz} establishes $\chi(KG_{n,k}) = n-2k+2$.

It is well-known that the Kneser graphs are regular with $\lambda_{max} = \left( \substack{n-k \\ k} \right)$ and $\lambda_{min} = - \left( \substack{n-k-1 \\ k-1} \right)$ \cite{godsil}.  Thus, our spectral bound gives almost surely
\[
\chi((KG_{n,k})_p) \geq \frac{\lmax (G)}{- \lmin (G) + (C/p) (\sqrt{\Delta} + \sqrt{\log(|V|)})} = \frac{{{n-k} \choose k}}{{{n-k-1} \choose k-1} + (C/p) \left[ \sqrt{{{n-k} \choose k}} + \sqrt{\log {{n} \choose k}} \right]}.
\]
For $k \geq 3$ (to avoid trivialities), the denomonitor is dominated by the first term, which gives almost surely
\[
\chi((KG_{n,k})_p) \geq \frac{{{n-k} \choose k}}{(1+\varepsilon_p) {{n-k-1} \choose k-1}} = \frac{n-k}{(1+\varepsilon_p)k},
\]
for some $0 < \varepsilon_p$ tending to $0$ as $n \to \infty$.

For $k \ll n$ and $p$ fixed, this gives a lower bound on $\chi((KG_{n,k})_p)$, which is on the order of $n$, which asymptotically matches the trivial upper bound $\chi(KG_{n,k})$.  Thus, this establishes Bukh's conjecture for Kneser graphs in this regime, and for sufficiently small values of $k$ (e.g., $k\geq3$ fixed) ours is the only general bound able to do this.  Although, for Kneser graphs in particular, $\chi((KG_{n,k})_p)$ is already well-understood for a wide range of $p$ by completely different methods \cite{chiKneser}.

%There are numerous other bounds on $\chi(G)$ depending on the eigenvalues of the adjacency matrix and/or the Laplacian---see \cite{wocjan}, \cite{fan}, \cite{vlad}, and \cite{naneh}. We have not observed that any of these gives a superior spectral bound on $\mathbb{E}[\chi(G_p)]$, but there may be particular $G$ for which they are superior to the bound we derive from Hoffman's result.
\vspace*{12pt}
Finally, we briefly turn our attention to a question of Shinkar, which we resolve negatively.  Hoping to use \eqref{shinkar equation} to resolve question \ref{bukh question} for all graphs, Shinkar  ~\cite{shinkar} asks the following:

%By \eqref{shinkar equation}, Shinkar ~\cite{shinkar} resolves Question \ref{bukh question} in the affirmative for all graphs $G$ having subgraphs $G'=(V',E') \subset G$ such that $\a(G') \leq C \cdot \frac{|V'|}{\chi(G)}$ for some $C>1$. He asks the following natural question which, if answered in the affirmative, would answer Bukh's question affirmatively in all cases:
\begin{question}[Shinkar]
Is it true that every graph $G$ contains an induced subgraph $G' \subset G$ such that $\chi(G') \geq c \cdot \chi(G)$, and $\a(G') \leq C \frac{|V(G')|}{\chi(G')}$ for some absolute constants $C,c > 0$?
\end{question}

%We show that the Kneser graphs provide a family of examples that answer Shinkar's question negatively. The \textbf{Kneser graph} $KG_{n,k}$ is a graph with $\binom{n}{k}$ vertices that represent the $k$-element subsets of an $n$-element set, where two vertices are adjacent if and only if the sets they represent are disjoint. For instance, $KG_{n,1}$ is the complete graph $K_n$. The independence number and chromatic number of Kneser graphs are well-studied: by Lov\'asz's Theorem ~\cite{lovasz} $\chi(KG_{n,k}) = n-2k+2$, and by the Erd\H{o}s-Ko-Rado Theorem ~\cite{erdos} $\a(KG_{n,k}) = \binom{n-1}{k-1}$.

The answer to this question is `no,' as shown by Kneser graphs.  Namely, Sudakov and Verstra\"ete ~\cite{sudakov} observe that if $H$ is any induced subgraph of $KG_{sk,k}$, then $|V(H)|/ \alpha(H) \leq s$.  This is because given $|V(H)|$ subsets of $\{1,2,\dots,sk\}$ of size $k$, by the pigeonhole principle there exists $i \in \{1,2,\dots,sk\}$ such that $i$ is contained in at least $k|V(H)|/sk$ of the sets of size $k$, and because these sets all intersect, the corresponding vertices form an indpendent set of size at least $|V(H)|/s$.  With this, we see that for sufficiently large $k$, the Kneser graphs $KG_{3k,k}$ provide an infinite family of counterexamples to Question 2.

%And since we also have $k \geq |KG_{sk,k}|/\a(KG_{sk,k}) = s$ by the Erd\H{o}s--Ko--Rado Theorem, we obtain $\iota(KG_{sk,k}) = s$.

%Now suppose that every Kneser graph $KG_{3k,k}$ satisfies the condition in Shinkar's question. Then in particular there are absolute constants $C,c > 0$ such that for every $k$ we have a subgraph $G_k \subset KG_{3k,k}$ satisfying $3 = \iota(KG_{3k,k}) \geq \frac{|V(G_k)|}{\a(G_k)} \geq \frac{1}{C} \chi(G_k) \geq \frac{c}{C} \chi(KG_{3k,k}) = k+2$. As $k \to \infty$, this is clearly false. We note that this example disproves a generalization of the negative statement of Shinkar's question, as it does not assume that the $G_k$ are induced.

%\nocite{*}
\bibliographystyle{siam}
\bibliography{bibliography}

\begin{appendices}
\section{Appendix: Proof of Proposition \ref{k4 planar minimizer}}
\begin{proof}[Proof of Proposition \ref{k4 planar minimizer}]
We will start by categorizing a particular family of graphs.  Define $\mathbb{F}_4$ as a collection of graphs such that each $G \in \mathbb{F}_4$ has exactly 4 triangles and satisfies the following two conditions:

\textbf{Condition 1}: For every  triangle $T \subset G$ and every vertex $p \in V(T)$, either 
\begin{enumerate}
    \item $p$ is not contained in any other triangle, or
    \item $p$ is contained in another triangle, and $T$ intersects some triangle $T'$ in an edge containing $p$
\end{enumerate}

The motivation for this condition is that if $G$ fails it, we can ``separate'' $G$ at $p$ as shown below, preserving the number of triangles, and leading to a graph $G'$ whose subgraph has a lower expected chromatic number (For every subgraph $H\subset G$, there is an equivalent subgraph $H' \subset G'$ obtained by separating $H$ at the same vertices where we split $G$.  Clearly, then any coloring of the vertices on $H$ can be copied onto $H'$, where separated vertices both share the same color as the original vertex.)

\begin{center}
\begin{tikzpicture}[node distance=1.7cm]
\node(1)[circle, draw, fill=blue] at (0,-.75) {};
\node(2)[circle, draw, fill=red] at (2,0) {p};
\node(3)[circle, draw, fill=green] at (0,.75) {};
\node(4)[circle, draw, fill=green] at (4,.75) {};
\node(5)[circle, draw, fill=blue] at (4,-.75) {};

\node(A)[] at (5,0) {};
\node(B)[] at (6,0) {};

\path (1) edge (2)
  (2) edge (3)
  (3) edge (1)
  (2) edge (4)
  (2) edge (5)
  (4) edge (5)
  ;
  
\draw[->, line width=.5mm] (A) edge (B);

\node(6)[circle, draw, fill=blue] at (7,-.75) {};
\node(7)[circle, draw, fill=red] at (9,0) {};
\node(8)[circle, draw, fill=green] at (7,.75) {};
\node(11)[circle, draw, fill=red] at (10,0) {};
\node(9)[circle, draw, fill=green] at (12,.75) {};
\node(10)[circle, draw, fill=blue] at (12,-.75) {};

\path (6) edge (7)
  (8) edge (7)
  (6) edge (8)
  (11) edge (10)
  (10) edge (9)
  (9) edge (11)
  ;
  
\end{tikzpicture}
\end{center}

\textbf{Condition 2}: For any graph $G \in \mathbb{F}_4$, $G$ has no proper subgraphs $H$ containing four triangles.
\vspace{.1cm}

As a result, we know that every edge in $G$ is an edge of some triangle of $G$.  Now consider the four triangles of $G$, $T_1$, $T_2$, $T_3$, and $T_4$.  $G$ is uniquely defined by how we identify the edges of each $T_i$ to the other triangles (remember, we just identify cannot identify individual points, as this may lead to a problem with our first condition).  Informally, we can construct $\mathbb{F}_4$ in the following manner: start with $T_1$, and let $\mathbb{G}_2$ be all the graphs obtained by identifying the edges of $T_2$ with the edges of $T_1$.  Construct $\mathbb{G}_3$ by identifying the edges of $T_3$ with the edges of each $G\in \mathbb{G}_2$.  Finally, construct $\mathbb{G}_4$ by identifying the edges of $T_4$ with the edges of each $G\in \mathbb{G}_3$.  There will certainly be graphs in $\mathbb{G}_4$ that do not have exactly four triangles; however, we can be sure that $\mathbb{F}_4 \subset \mathbb{G}_4$.

Some brief observations that will make our constructions of the $\mathbb{G}_i$ easier:
\begin{itemize}
    \item For any given triangles $T_i$ and $T_j$, we can only identify at most one of the edges from each triangle.  If $T_i$ and $T_j$ share two edges, they must share all three edges, and would therefore be the same triangle. We can ignore these cases, as we wish for the four $T_k$ to represent four distinct triangles in the identification graphs.
    \item The process of identifying the edges of the $T_i$ in turn is commutative.  Therefore, if our final graph has $n$ components, we can choose the order of identification such that if the edges of $T_j$ are not identified to any $T_i$ for $i < j$, then the edges of all $T_k$, $k > j$ will also not be identified to any $T_i$.  In other words, we can always choose to have $T_j$'s edges only be identified to the edges of exactly one component of $H \in \mathbb{G}_{j-1}$
\end{itemize}

Now, we can start constructing $\mathbb{G}_1$, $\mathbb{G}_2$, $\mathbb{G}_3$ and $\mathbb{G}_4$.  The first two are trivial.

\vspace{.5cm}
\fbox{\begin{minipage}{1.1em}
$\mathbb{G}_1$
\end{minipage}}

\begin{center}
\begin{tikzpicture}[node distance = .6cm]
\node(1)[circle, draw, fill=pink] at (-.2,0) {};
\node(2)[circle, draw, fill=pink][below of = 1]{};
\node(3)[circle, draw, fill=pink][right of = 2]{};

\draw(1)[magenta, very thick]--(2);
\draw(2)[magenta, very thick]--(3);
\draw(3)[magenta, very thick]--(1);

\node[draw,align=left] at (0,-1.4) {$A_1$};

\end{tikzpicture}
\end{center}

\vspace{.5cm}
\fbox{\begin{minipage}{1.1em}
$\mathbb{G}_2$
\end{minipage}}
\begin{center}
\begin{tikzpicture}[node distance = .6cm]
\node(1)[circle, draw, fill=pink] at (-.2,0) {};
\node(2)[circle, draw, fill=pink][below of = 1]{};
\node(3)[circle, draw, fill=pink][right of = 2]{};
\node(13)[right of = 1]{};
\node(4)[circle, draw, fill=pink][right of = 13]{};
\node(5)[circle, draw, fill=pink][below of = 4]{};
\node(6)[circle, draw, fill=pink][right of = 5]{};

\draw(1)[magenta, very thick]--(2);
\draw(2)[magenta, very thick]--(3);
\draw(3)[magenta, very thick]--(1);
\draw(4)[magenta, very thick]--(5);
\draw(5)[magenta, very thick]--(6);
\draw(6)[magenta, very thick]--(4);

\node[draw,align=left] at (.7,-1.4) {$B_1$};

\end{tikzpicture}
\hspace{.6cm}
\begin{tikzpicture}[node distance = .6cm]
\node(13){};
\node(3)[circle, draw, fill=pink][below of = 13]{};
\node(4)[circle, draw, fill=pink][right of = 13]{};
\node(5)[circle, draw, fill=pink][below of = 4]{};
\node(6)[circle, draw, fill=pink][right of = 5]{};

\draw(3)[magenta, very thick]--(4);
\draw(3)[magenta, very thick]--(5);
\draw(4)[magenta, very thick]--(5);
\draw(5)[magenta, very thick]--(6);
\draw(6)[magenta, very thick]--(4);

\node[draw,align=left] at (.7,-1.4) {$B_2$};

\end{tikzpicture}
\end{center}

\vspace{.5cm}
\fbox{\begin{minipage}{1.1em}
$\mathbb{G}_3$
\end{minipage}}
\vspace{.5cm}

Consider $B_1$.  We can construct exactly two distinct (up to isomorphism) child graphs, by either identifying none of the edges of $T_3$, or identifying one of the edges of $T_3$ to one of the component triangles.  We cannot do anything more, as this would result in identifying two edges to the same triangle or connecting to separate components:

\begin{center}
\begin{tikzpicture}[node distance = .6cm]
\node(1)[circle, draw, fill=pink] at (-.2,0) {};
\node(2)[circle, draw, fill=pink][below of = 1]{};
\node(3)[circle, draw, fill=pink][right of = 2]{};
\node(13)[right of = 1]{};
\node(4)[circle, draw, fill=pink][right of = 13]{};
\node(5)[circle, draw, fill=pink][below of = 4]{};
\node(6)[circle, draw, fill=pink][right of = 5]{};
\node(14)[right of = 4]{};
\node(7)[circle, draw, fill=pink][right of = 14]{};
\node(8)[circle, draw, fill=pink][below of = 7]{};
\node(9)[circle, draw, fill=pink][right of = 8]{};

\draw(1)[magenta, very thick]--(2);
\draw(2)[magenta, very thick]--(3);
\draw(3)[magenta, very thick]--(1);
\draw(4)[magenta, very thick]--(5);
\draw(5)[magenta, very thick]--(6);
\draw(6)[magenta, very thick]--(4);
\draw(7)[magenta, very thick]--(8);
\draw(8)[magenta, very thick]--(9);
\draw(9)[magenta, very thick]--(7);

\node[draw,align=left] at (1.3,-1.4) {$C_1$};

\end{tikzpicture}
\hspace{.6cm}
\begin{tikzpicture}[node distance = .6cm]
\node(13){};
\node(3)[circle, draw, fill=pink][below of = 13]{};
\node(4)[circle, draw, fill=pink][right of = 13]{};
\node(5)[circle, draw, fill=pink][below of = 4]{};
\node(6)[circle, draw, fill=pink][right of = 5]{};
\node(14)[right of = 4]{};
\node(7)[circle, draw, fill=pink][right of = 14]{};
\node(8)[circle, draw, fill=pink][below of = 7]{};
\node(9)[circle, draw, fill=pink][right of = 8]{};

\draw(3)[magenta, very thick]--(4);
\draw(3)[magenta, very thick]--(5);
\draw(4)[magenta, very thick]--(5);
\draw(5)[magenta, very thick]--(6);
\draw(6)[magenta, very thick]--(4);
\draw(7)[magenta, very thick]--(8);
\draw(8)[magenta, very thick]--(9);
\draw(9)[magenta, very thick]--(7);

\node[draw,align=left] at (1.3,-1.4) {$C_2$};

\end{tikzpicture}
\end{center}

Now consider the second graph of $\mathbb{G}_2$, $B_2$.  If we do not identify $T_3$ to any edge, then we get a graph isomorphic to $H_2$. If we identify one edge of $T_3$ to any of the four external edges, we will obtain the same graph (up to isomorphism):
\begin{center}
\begin{tikzpicture}[node distance = .7cm]
\node(13){};
\node(3)[circle, draw, fill=pink][below of = 13]{};
\node(4)[circle, draw, fill=pink] at (.95,0){};
\node(5)[circle, draw, fill=pink][right of = 3]{};
\node(6)[circle, draw, fill=pink][right of = 5]{};
\node(9)[circle, draw, fill=pink][right of = 6]{};

\draw(3)[magenta, very thick]--(4);
\draw(3)[magenta, very thick]--(5);
\draw(4)[magenta, very thick]--(5);
\draw(4)[magenta, very thick]--(9);
\draw(5)[magenta, very thick]--(6);
\draw(6)[magenta, very thick]--(4);
\draw(9)[magenta, very thick]--(6);

\node[draw,align=left] at (1.0,-1.4) {$C_3$};
\end{tikzpicture}
\end{center}
If we identify one edge of $T_3$ to the center edge of $B_1$, then we obtain the following graph:
\begin{center}
\begin{tikzpicture}[node distance = .8cm]
\node(13){};
\node(3)[circle, draw, fill=pink] at (0,-.4){};
\node(4)[circle, draw, fill=pink] at (.5,0){};
\node(5)[circle, draw, fill=pink][below of = 4]{};
\node(6)[circle, draw, fill=pink] at (.9,-.4){};
\node(9)[circle, draw, fill=pink] at (1.7,-.4){};

\draw(3)[magenta, very thick]--(4);
\draw(3)[magenta, very thick]--(5);
\draw(4)[magenta, very thick]--(5);
\draw(4)[magenta, very thick]--(9);
\draw(5)[magenta, very thick]--(6);
\draw(6)[magenta, very thick]--(4);
\draw(5)[magenta, very thick]--(9);

\node[draw,align=left] at (.7,-1.4) {$C_4$};
\end{tikzpicture}
\end{center}
Finally, suppose we identify two edges of $T_3$ to two edges of $B_2$.  We cannot identify them to two edges from the same triangle in $B_2$.  Therefore, we can identify neither of the two edges to the center edge of $B_2$, as any other edge would lie on the same triangle as the center edge.  Therefore, we have two options: identify two edge, one from each triangle, that are adjacent, or non adjacent. First, a larger visual:  

\begin{center}
\begin{tikzpicture}[node distance = 1cm]
\node(13){};
\node(3)[circle, draw, fill=yellow][below of = 13]{A};
\node(4)[circle, draw, fill=yellow][right of = 13]{B};
\node(5)[circle, draw, fill=yellow][below of = 4]{C};
\node(6)[circle, draw, fill=yellow][right of = 5]{D};

\node(7)[circle, draw, fill=yellow][right of = 6]{1};
\node(8)[circle, draw, fill=yellow][above of = 7]{2};
\node(9)[circle, draw, fill=yellow][right of = 7]{3};

\node[draw,align=left] at (3.5,-2) {$T_3$};
\node[draw,align=left] at (1,-2) {$B_2$};

\draw(3)[black, very thick]--(4);
\draw(3)[black, very thick]--(5);
\draw(4)[black, very thick]--(5);
\draw(5)[black, very thick]--(6);
\draw(6)[black, very thick]--(4);
\draw(7)[black, very thick]--(8);
\draw(8)[black, very thick]--(9);
\draw(7)[black, very thick]--(9);

\end{tikzpicture}
\end{center}

Without loss of generality, assume that the two edges we are identifying from $T_3$ are \{1,2\} and \{2,3\}.  Assume we identify these two edges with the edges \{A,B\} and \{B,D\}.  We must do this identification by identifying 2 to B, and we are left with the following graph:
\begin{center}
\begin{tikzpicture}[node distance = .6cm]
\node(13)[circle, draw, fill=pink] at (.5, 0) {};
\node(3)[circle, draw, fill=pink] at (0, -.8) {};
\node(4)[circle, draw, fill=pink] at (1, -.8) {};
\node(5)[circle, draw, fill=pink] at (.5, -.5) {};

\draw(3)[magenta, very thick]--(4);
\draw(3)[magenta, very thick]--(5);
\draw(3)[magenta, very thick]--(13);
\draw(4)[magenta, very thick]--(5);
\draw(5)[magenta, very thick]--(13);
\draw(13)[magenta, very thick]--(4);
\draw(5)[magenta, very thick]--(13);

\node[draw,align=left] at (.6,-1.4) {$C_5$};

\end{tikzpicture}
\end{center}

(Note, the graph $C_5$ has four triangles. However, one of these triangles is not identified to any of $T_1$, $T_2$, or $T_3$, but is composed of one edge from each.)

Now suppose that we identify the same two edges of $T_3$ to the edges \{A,B\} and \{C,D\}.  In whatever manner we choose to identify the individual vertices, we will have to identify either vertex C or D with vertex A or B.  This would necessarily result in the destruction of one of the previous triangles in $B_2$.  Therefore, we cannot identify any two edges from $T_3$ to non-adjacent edges on $B_2$.

Therefore, we have completely categorized $\mathbb{G}_3$.

\vspace{.5cm}
\fbox{\begin{minipage}{1.1em}
$\mathbb{G}_4$
\end{minipage}}
\vspace{.5cm}

In the same manner as before, we can generate the first two graphs of $\mathbb{G}_4$ from $C_1$:

\begin{center}
\begin{tikzpicture}[node distance = .6cm]
\node(1)[circle, draw, fill=pink] at (-.2,0) {};
\node(2)[circle, draw, fill=pink][below of = 1]{};
\node(3)[circle, draw, fill=pink][right of = 2]{};
\node(13)[right of = 1]{};
\node(4)[circle, draw, fill=pink][right of = 13]{};
\node(5)[circle, draw, fill=pink][below of = 4]{};
\node(6)[circle, draw, fill=pink][right of = 5]{};
\node(14)[right of = 4]{};
\node(7)[circle, draw, fill=pink][right of = 14]{};
\node(8)[circle, draw, fill=pink][below of = 7]{};
\node(9)[circle, draw, fill=pink][right of = 8]{};
\node(15)[right of = 7]{};
\node(10)[circle, draw, fill=pink][right of = 15]{};
\node(11)[circle, draw, fill=pink][below of = 10]{};
\node(12)[circle, draw, fill=pink][right of = 11]{};

\draw(1)[magenta, very thick]--(2);
\draw(2)[magenta, very thick]--(3);
\draw(3)[magenta, very thick]--(1);
\draw(4)[magenta, very thick]--(5);
\draw(5)[magenta, very thick]--(6);
\draw(6)[magenta, very thick]--(4);
\draw(7)[magenta, very thick]--(8);
\draw(8)[magenta, very thick]--(9);
\draw(9)[magenta, very thick]--(7);
\draw(10)[magenta, very thick]--(11);
\draw(11)[magenta, very thick]--(12);
\draw(12)[magenta, very thick]--(10);

\node[draw,align=left] at (1.8,-1.2) {$D_1$};

\end{tikzpicture}
\hspace{.6cm}
\begin{tikzpicture}[node distance = .6cm]
\node(13){};
\node(3)[circle, draw, fill=pink][below of = 13]{};
\node(4)[circle, draw, fill=pink][right of = 13]{};
\node(5)[circle, draw, fill=pink][below of = 4]{};
\node(6)[circle, draw, fill=pink][right of = 5]{};
\node(14)[right of = 4]{};
\node(7)[circle, draw, fill=pink][right of = 14]{};
\node(8)[circle, draw, fill=pink][below of = 7]{};
\node(9)[circle, draw, fill=pink][right of = 8]{};
\node(15)[right of = 7]{};
\node(10)[circle, draw, fill=pink][right of = 15]{};
\node(11)[circle, draw, fill=pink][below of = 10]{};
\node(12)[circle, draw, fill=pink][right of = 11]{};

\draw(3)[magenta, very thick]--(4);
\draw(3)[magenta, very thick]--(5);
\draw(4)[magenta, very thick]--(5);
\draw(5)[magenta, very thick]--(6);
\draw(6)[magenta, very thick]--(4);
\draw(7)[magenta, very thick]--(8);
\draw(8)[magenta, very thick]--(9);
\draw(9)[magenta, very thick]--(7);
\draw(10)[magenta, very thick]--(11);
\draw(11)[magenta, very thick]--(12);
\draw(12)[magenta, very thick]--(10);

\node[draw,align=left] at (1.8,-1.2) {$D_2$};

\end{tikzpicture}
\end{center}

Consider $C_2$.  If $T_4$ is disjoint, then we obtain the same graph as $D_2$.  By the same reasoning as before, if we identify one edge of $T_4$ to an edge of the triangular component of $C_2$, one edge to an external edge of the larger component of $C_2$, one edge to the internal edge of the larger component of $C_2$, or two edges (in the only way possible) to the larger component of $C_2$, we obtain, respectively:

\begin{center}
\begin{tikzpicture}[node distance = .8cm]
\node(13){};
\node(3)[circle, draw, fill=pink][below of = 13]{};
\node(4)[circle, draw, fill=pink][right of = 13]{};
\node(5)[circle, draw, fill=pink][below of = 4]{};
\node(6)[circle, draw, fill=pink][right of = 5]{};
\node(14)[right of = 4]{};

\node(8)[circle, draw, fill=pink][right of = 6]{};
\node(9)[circle, draw, fill=pink][right of = 8]{};
\node(7)[circle, draw, fill=pink][above of = 9]{};
\node(10)[circle, draw, fill=pink][right of = 9]{};

\draw(3)[magenta, very thick]--(4);
\draw(3)[magenta, very thick]--(5);
\draw(4)[magenta, very thick]--(5);
\draw(5)[magenta, very thick]--(6);
\draw(6)[magenta, very thick]--(4);
\draw(7)[magenta, very thick]--(8);
\draw(8)[magenta, very thick]--(9);
\draw(9)[magenta, very thick]--(7);
\draw(10)[magenta, very thick]--(9);
\draw(10)[magenta, very thick]--(7);

\node[draw,align=left] at (1.6,-1.5) {$D_3$};

\end{tikzpicture}
\hspace{.6cm}
\begin{tikzpicture}[node distance = .8cm]
\node(13){};
\node(3)[circle, draw, fill=pink][below of = 13]{};
\node(4)[circle, draw, fill=pink][right of = 13]{};
\node(5)[circle, draw, fill=pink][below of = 4]{};
\node(6)[circle, draw, fill=pink][right of = 5]{};
\node(9)[circle, draw, fill=pink][above of = 6]{};
\node(10)[circle, draw, fill=pink][right of = 6]{};
\node(11)[circle, draw, fill=pink][above of = 10]{};
\node(12)[circle, draw, fill=pink][right of = 10]{};

\draw(3)[magenta, very thick]--(4);
\draw(3)[magenta, very thick]--(5);
\draw(4)[magenta, very thick]--(5);
\draw(4)[magenta, very thick]--(9);
\draw(5)[magenta, very thick]--(6);
\draw(6)[magenta, very thick]--(4);
\draw(9)[magenta, very thick]--(6);
\draw(10)[magenta, very thick]--(11);
\draw(11)[magenta, very thick]--(12);
\draw(12)[magenta, very thick]--(10);

\node[draw,align=left] at (1.4,-1.5) {$D_4$};

\end{tikzpicture}

\vspace{.6cm}
\begin{tikzpicture}[node distance = .8cm]
\node(13){};
\node(3)[circle, draw, fill=pink] at (0,-.4){};
\node(4)[circle, draw, fill=pink] at (.5,0){};
\node(5)[circle, draw, fill=pink][below of = 4]{};
\node(6)[circle, draw, fill=pink] at (.9,-.4){};
\node(9)[circle, draw, fill=pink] at (1.7,-.4){};
\node(10)[circle, draw, fill=pink] at (2.3, -.8){};
\node(11)[circle, draw, fill=pink][above of = 10]{};
\node(12)[circle, draw, fill=pink][right of = 10]{};

\draw(3)[magenta, very thick]--(4);
\draw(3)[magenta, very thick]--(5);
\draw(4)[magenta, very thick]--(5);
\draw(4)[magenta, very thick]--(9);
\draw(5)[magenta, very thick]--(6);
\draw(6)[magenta, very thick]--(4);
\draw(5)[magenta, very thick]--(9);
\draw(10)[magenta, very thick]--(11);
\draw(11)[magenta, very thick]--(12);
\draw(12)[magenta, very thick]--(10);

\node[draw,align=left] at (1.4,-1.4) {$D_5$};
\end{tikzpicture}
\hspace{.5cm}
\begin{tikzpicture}[node distance = .8cm]
\node(13)[circle, draw, fill=pink] at (.5, 0) {};
\node(3)[circle, draw, fill=pink] at (0, -.8) {};
\node(4)[circle, draw, fill=pink] at (1, -.8) {};
\node(5)[circle, draw, fill=pink] at (.5, -.5) {};

\node(7)[circle, draw, fill=pink] at (1.7, -.8) {};
\node(6)[circle, draw, fill=pink][above of = 7]{};
\node(8)[circle, draw, fill=pink][right of = 7]{};

\draw(3)[magenta, very thick]--(4);
\draw(3)[magenta, very thick]--(5);
\draw(3)[magenta, very thick]--(13);
\draw(4)[magenta, very thick]--(5);
\draw(5)[magenta, very thick]--(13);
\draw(13)[magenta, very thick]--(4);
\draw(5)[magenta, very thick]--(13);
\draw(6)[magenta, very thick]--(7);
\draw(7)[magenta, very thick]--(8);
\draw(8)[magenta, very thick]--(6);

\node[draw,align=left] at (1.4,-1.4) {$D_6$};

\end{tikzpicture}

\end{center}

Consider $C_3$.  If we add a disjoint triangle, we will obtain $D_4$ again.  If we identify one edge of $T_4$ to the exterior edge of $C_3$ lying in the center triangle, to one of other exterior edges, or to an interior edge of $C_3$, we will obtain, respectively:

\begin{center}
\begin{tikzpicture}[node distance = .6cm]
\node(13)[circle, draw, fill=pink] at (.7, 0) {};
\node(3)[circle, draw, fill=pink] at (2, 0) {};
\node(4)[circle, draw, fill=pink] at (1.35, 0) {};
\node(5)[circle, draw, fill=pink] at (1, -.5) {};
\node(6)[circle, draw, fill=pink] at (1.7, -.5) {};
\node(9)[circle, draw, fill=pink] at (1.35, -1) {};

\draw(3)[magenta, very thick]--(4);
\draw(3)[magenta, very thick]--(6);
\draw(4)[magenta, very thick]--(5);
\draw(6)[magenta, very thick]--(9);
\draw(5)[magenta, very thick]--(6);
\draw(6)[magenta, very thick]--(4);
\draw(9)[magenta, very thick]--(5);
\draw(13)[magenta, very thick]--(4);
\draw(5)[magenta, very thick]--(13);

\node[draw,align=left] at (1.4,-1.7) {$D_7$};

\end{tikzpicture}
\hspace{.5cm}
\begin{tikzpicture}[node distance = .6cm]
\node(13) at (.5,0) {};
\node(3)[circle, draw, fill=pink][below of = 13]{};
\node(4)[circle, draw, fill=pink][right of = 13]{};
\node(5)[circle, draw, fill=pink][below of = 4]{};
\node(6)[circle, draw, fill=pink][right of = 5]{};
\node(9)[circle, draw, fill=pink][above of = 6]{};
\node(12)[circle, draw, fill=pink][right of = 6]{};

\draw(3)[magenta, very thick]--(4);
\draw(3)[magenta, very thick]--(5);
\draw(4)[magenta, very thick]--(5);
\draw(4)[magenta, very thick]--(9);
\draw(5)[magenta, very thick]--(6);
\draw(6)[magenta, very thick]--(4);
\draw(9)[magenta, very thick]--(6);
\draw(12)[magenta, very thick]--(6);
\draw(12)[magenta, very thick]--(9);

\node[draw,align=left] at (1.4,-1.2) {$D_8$};

\end{tikzpicture}
\hspace{.5cm}
\begin{tikzpicture}[node distance = .6cm]
\node(13)[circle, draw, fill=pink] at (-.5, 0) {};
\node(3)[circle, draw, fill=pink][below of = 13]{};
\node(4)[circle, draw, fill=pink][right of = 13]{};
\node(5)[circle, draw, fill=pink][below of = 4]{};
\node(6)[circle, draw, fill=pink][right of = 5]{};
\node(9)[circle, draw, fill=pink][above of = 6]{};

\draw(3)[magenta, very thick]--(4);
\draw(3)[magenta, very thick]--(5);
\draw(4)[magenta, very thick]--(5);
\draw(4)[magenta, very thick]--(9);
\draw(5)[magenta, very thick]--(6);
\draw(6)[magenta, very thick]--(4);
\draw(9)[magenta, very thick]--(5);
\draw(13)[magenta, very thick]--(4);
\draw(3)[magenta, very thick]--(13);

\node[draw,align=left] at (0,-1.2) {$D_9$};

\end{tikzpicture}
\end{center}

Suppose we identify two edges of $T_4$ to $C_3$:

\begin{center}
\begin{tikzpicture}[node distance = 1cm]
\node(13){};
\node(3)[circle, draw, fill=yellow][below of = 13]{A};
\node(4)[circle, draw, fill=yellow] at (.95,0){B};
\node(5)[circle, draw, fill=yellow][right of = 3]{C};
\node(6)[circle, draw, fill=yellow][right of = 5]{D};
\node(9)[circle, draw, fill=yellow][right of = 6]{E};

\draw(3)[black, very thick]--(4);
\draw(3)[black, very thick]--(5);
\draw(4)[black, very thick]--(5);
\draw(4)[black, very thick]--(9);
\draw(5)[black, very thick]--(6);
\draw(6)[black, very thick]--(4);
\draw(9)[black, very thick]--(6);

\node(10)[circle, draw, fill=yellow][right of = 9]{1};
\node(11)[circle, draw, fill=yellow][above of = 10]{2};
\node(12)[circle, draw, fill=yellow][right of = 10]{3};

\draw(10)[black, very thick]--(11);
\draw(11)[black, very thick]--(12);
\draw(12)[black, very thick]--(10);

\node[draw,align=left] at (4.5,-2) {$T_4$};
\node[draw,align=left] at (1,-2) {$C_3$};

\end{tikzpicture}
\end{center}

Without loss of generality let the two edges of $T_4$ be \{1,2\} and \{2,3\}.  Our two edges on $C_3$ either i) lie on a subgraph isomorphic to $B_2$ or ii) one edge is \{A,B\} or \{A,C\}, and the other is \{B,E\} or \{D,E\}.  In case i), WLOG let the subgraph of $C_3$ be the induced subgraph on vertices C, B, D, and E.  We know from our earlier argument that we must either identify \{1,2\} and \{2,3\} to \{C,B\} and \{B,E\}, or \{C,D\} and \{D,E\}.  With either choice, we obtain the graph:

\begin{center}
\begin{tikzpicture}[node distance = .6cm]
\node(13)[circle, draw, fill=pink] at (0, .25) {};
\node(3)[circle, draw, fill=pink] at (0, -.8) {};
\node(4)[circle, draw, fill=pink] at (1, -.8) {};
\node(5)[circle, draw, fill=pink] at (.25, -.4) {};

\node(6)[circle, draw, fill=pink][left of = 3]{};

\draw(3)[magenta, very thick]--(4);
\draw(3)[magenta, very thick]--(5);
\draw(3)[magenta, very thick]--(13);
\draw(4)[magenta, very thick]--(5);
\draw(5)[magenta, very thick]--(13);
\draw(13)[magenta, very thick]--(4);
\draw(5)[magenta, very thick]--(13);

\draw(6)[magenta, very thick]--(3);
\draw(6)[magenta, very thick]--(13);

\node[draw,align=left] at (.3,-1.4) {$D_{10}$};

\end{tikzpicture}
\end{center}

In case ii), if we identify \{1,2\} to \{E,B\} and \{2,3\} to \{B,A\}, we obtain the following graph:
\begin{center}
\begin{tikzpicture}[node distance = .6cm]
\node(1)[circle, draw, fill=pink] at (1, .2) {};
\node(2)[circle, draw, fill=pink] at (1.8, .2) {};
\node(5)[circle, draw, fill=pink] at (1.4, -.2) {};
\node(4)[circle, draw, fill=pink] at (1, -.6) {};
\node(3)[circle, draw, fill=pink] at (1.8, -.6) {};

\draw(1)[magenta, very thick]--(2);
\draw(2)[magenta, very thick]--(3);
\draw(3)[magenta, very thick]--(4);
\draw(4)[magenta, very thick]--(1);
\draw(1)[magenta, very thick]--(5);
\draw(2)[magenta, very thick]--(5);
\draw(3)[magenta, very thick]--(5);
\draw(4)[magenta, very thick]--(5);

\node[draw,align=left] at (1.4,-1.2) {$D_{11}$};

\end{tikzpicture}
\end{center}

If we identify \{1,2\} to \{A,B\} and \{2,3\} to \{D,E\} (equivalent to identifying \{1,2\} to \{A,C\} and \{2,3\} to \{B,E\}), then we must identify vertex 2 to A and E (otherwise, we will collapse two triangles).  At this point, the graph we will obtain will contain a copy of $K_4$, so we need not consider it for $\mathbb{F}_4$.

If we identify \{1,2\} to \{A,C\} and \{2,3\} to \{D,E\}, then we must identify vertex 2 to A and E again.  By the same logic as before, we can ignore this graph.

Consider the graph $C_4$.  If $T_4$ is disjoint, then we obtain $D_5$.  If we identify one edge of $T_4$ to one of the exterior six edges, we obtain $D_9$.  If we identify the edge of $T_4$ to the central edge shared by the other three triangles, we obtain the graph:

\begin{center}
\begin{tikzpicture}[node distance = .6cm]
\node(13)[circle, draw, fill=pink] at (.7, 0) {};
\node(3)[circle, draw, fill=pink][below of = 13]{};
\node(4)[circle, draw, fill=pink][right of = 13]{};
\node(5)[circle, draw, fill=pink][below of = 4]{};
\node(6)[circle, draw, fill=pink][right of = 5]{};
\node(9)[circle, draw, fill=pink][above of = 6]{};

\draw(3)[magenta, very thick]--(4);
\draw(3)[magenta, very thick]--(5);
\draw(4)[magenta, very thick]--(5);
\draw(4)[magenta, very thick]--(9);
\draw(5)[magenta, very thick]--(6);
\draw(6)[magenta, very thick]--(4);
\draw(9)[magenta, very thick]--(5);
\draw(13)[magenta, very thick]--(4);
\draw(5)[magenta, very thick]--(13);

\node[draw,align=left] at (1.4,-1.2) {$D_{12}$};
\end{tikzpicture}
\end{center}

If we identify two edges of $T_4$, then we must identify them to edges lying in a subgraph of $C_4$ isomorphic to $B_2$.  Therefore, the graph we obtain must have a copy of $K_4$.

Finally, consider $C_5$.  Identifying edges in $T_4$ to edge in $C_5$ in such a way such that we preserve the uniqueness of all the triangles will force us to have a copy of $K_4$ in our resulting graph.  Therefore, we can ignore these graphs.  For formality's sake, we should identify $T_4$ with the triangle in $C_5$ comprised of one edge from each of the previous three $T_i$.  We then obtain the graph $D_{13}$ (which is the same graph as $C_5$):

\begin{center}
\begin{tikzpicture}[node distance = .6cm]
\node(13)[circle, draw, fill=pink] at (.5, 0) {};
\node(3)[circle, draw, fill=pink] at (0, -.8) {};
\node(4)[circle, draw, fill=pink] at (1, -.8) {};
\node(5)[circle, draw, fill=pink] at (.5, -.5) {};

\draw(3)[magenta, very thick]--(4);
\draw(3)[magenta, very thick]--(5);
\draw(3)[magenta, very thick]--(13);
\draw(4)[magenta, very thick]--(5);
\draw(5)[magenta, very thick]--(13);
\draw(13)[magenta, very thick]--(4);
\draw(5)[magenta, very thick]--(13);

\node[draw,align=left] at (.6,-1.4) {$D_{13}$};

\end{tikzpicture}
\end{center}

Now that we have categorized $\mathbb{G}_4$, we can pull $\mathbb{F}_4$ as a subset from these graphs.  We see that $D_6$ and $D_{10}$ contain copies of $D_{13}$ ($K_4$).  However, every other graph satisfies the properties for $\mathbb{F}_4$.  Therefore, the graphs of $\mathbb{F}_4$ are as follows:

\begin{center}

\begin{tikzpicture}[node distance = .6cm]
\node(1)[circle, draw, fill=pink] at (-.2,0) {};
\node(2)[circle, draw, fill=pink][below of = 1]{};
\node(3)[circle, draw, fill=pink][right of = 2]{};
\node(13)[right of = 1]{};
\node(4)[circle, draw, fill=pink][right of = 13]{};
\node(5)[circle, draw, fill=pink][below of = 4]{};
\node(6)[circle, draw, fill=pink][right of = 5]{};
\node(14)[right of = 4]{};
\node(7)[circle, draw, fill=pink][right of = 14]{};
\node(8)[circle, draw, fill=pink][below of = 7]{};
\node(9)[circle, draw, fill=pink][right of = 8]{};
\node(15)[right of = 7]{};
\node(10)[circle, draw, fill=pink][right of = 15]{};
\node(11)[circle, draw, fill=pink][below of = 10]{};
\node(12)[circle, draw, fill=pink][right of = 11]{};

\draw(1)[magenta, very thick]--(2);
\draw(2)[magenta, very thick]--(3);
\draw(3)[magenta, very thick]--(1);
\draw(4)[magenta, very thick]--(5);
\draw(5)[magenta, very thick]--(6);
\draw(6)[magenta, very thick]--(4);
\draw(7)[magenta, very thick]--(8);
\draw(8)[magenta, very thick]--(9);
\draw(9)[magenta, very thick]--(7);
\draw(10)[magenta, very thick]--(11);
\draw(11)[magenta, very thick]--(12);
\draw(12)[magenta, very thick]--(10);

\node[draw,align=left] at (1.8,-1.2) {$D_1$};

\end{tikzpicture}
\hspace{.6cm}
\begin{tikzpicture}[node distance = .6cm]
\node(13){};
\node(3)[circle, draw, fill=pink][below of = 13]{};
\node(4)[circle, draw, fill=pink][right of = 13]{};
\node(5)[circle, draw, fill=pink][below of = 4]{};
\node(6)[circle, draw, fill=pink][right of = 5]{};
\node(14)[right of = 4]{};
\node(7)[circle, draw, fill=pink][right of = 14]{};
\node(8)[circle, draw, fill=pink][below of = 7]{};
\node(9)[circle, draw, fill=pink][right of = 8]{};
\node(15)[right of = 7]{};
\node(10)[circle, draw, fill=pink][right of = 15]{};
\node(11)[circle, draw, fill=pink][below of = 10]{};
\node(12)[circle, draw, fill=pink][right of = 11]{};

\draw(3)[magenta, very thick]--(4);
\draw(3)[magenta, very thick]--(5);
\draw(4)[magenta, very thick]--(5);
\draw(5)[magenta, very thick]--(6);
\draw(6)[magenta, very thick]--(4);
\draw(7)[magenta, very thick]--(8);
\draw(8)[magenta, very thick]--(9);
\draw(9)[magenta, very thick]--(7);
\draw(10)[magenta, very thick]--(11);
\draw(11)[magenta, very thick]--(12);
\draw(12)[magenta, very thick]--(10);

\node[draw,align=left] at (1.8,-1.2) {$D_2$};

\end{tikzpicture}

\vspace{.5cm}
\begin{tikzpicture}[node distance = .6cm]
\node(13){};
\node(3)[circle, draw, fill=pink][below of = 13]{};
\node(4)[circle, draw, fill=pink][right of = 13]{};
\node(5)[circle, draw, fill=pink][below of = 4]{};
\node(6)[circle, draw, fill=pink][right of = 5]{};
\node(14)[right of = 4]{};

\node(8)[circle, draw, fill=pink][right of = 6]{};
\node(9)[circle, draw, fill=pink][right of = 8]{};
\node(7)[circle, draw, fill=pink][above of = 9]{};
\node(10)[circle, draw, fill=pink][right of = 9]{};

\draw(3)[magenta, very thick]--(4);
\draw(3)[magenta, very thick]--(5);
\draw(4)[magenta, very thick]--(5);
\draw(5)[magenta, very thick]--(6);
\draw(6)[magenta, very thick]--(4);
\draw(7)[magenta, very thick]--(8);
\draw(8)[magenta, very thick]--(9);
\draw(9)[magenta, very thick]--(7);
\draw(10)[magenta, very thick]--(9);
\draw(10)[magenta, very thick]--(7);

\node[draw,align=left] at (1.8,-1.2) {$D_3$};

\end{tikzpicture}
\hspace{.5cm}
\begin{tikzpicture}[node distance = .6cm]
\node(13){};
\node(3)[circle, draw, fill=pink][below of = 13]{};
\node(4)[circle, draw, fill=pink][right of = 13]{};
\node(5)[circle, draw, fill=pink][below of = 4]{};
\node(6)[circle, draw, fill=pink][right of = 5]{};
\node(9)[circle, draw, fill=pink][above of = 6]{};
\node(10)[circle, draw, fill=pink][right of = 6]{};
\node(11)[circle, draw, fill=pink][above of = 10]{};
\node(12)[circle, draw, fill=pink][right of = 10]{};

\draw(3)[magenta, very thick]--(4);
\draw(3)[magenta, very thick]--(5);
\draw(4)[magenta, very thick]--(5);
\draw(4)[magenta, very thick]--(9);
\draw(5)[magenta, very thick]--(6);
\draw(6)[magenta, very thick]--(4);
\draw(9)[magenta, very thick]--(6);
\draw(10)[magenta, very thick]--(11);
\draw(11)[magenta, very thick]--(12);
\draw(12)[magenta, very thick]--(10);

\node[draw,align=left] at (1.4,-1.2) {$D_4$};

\end{tikzpicture}
\hspace{.5cm}
\begin{tikzpicture}[node distance = .8cm]
\node(13){};
\node(3)[circle, draw, fill=pink] at (0,-.4){};
\node(4)[circle, draw, fill=pink] at (.5,0){};
\node(5)[circle, draw, fill=pink][below of = 4]{};
\node(6)[circle, draw, fill=pink] at (.9,-.4){};
\node(9)[circle, draw, fill=pink] at (1.7,-.4){};
\node(10)[circle, draw, fill=pink] at (2.3, -.8){};
\node(11)[circle, draw, fill=pink][above of = 10]{};
\node(12)[circle, draw, fill=pink][right of = 10]{};

\draw(3)[magenta, very thick]--(4);
\draw(3)[magenta, very thick]--(5);
\draw(4)[magenta, very thick]--(5);
\draw(4)[magenta, very thick]--(9);
\draw(5)[magenta, very thick]--(6);
\draw(6)[magenta, very thick]--(4);
\draw(5)[magenta, very thick]--(9);
\draw(10)[magenta, very thick]--(11);
\draw(11)[magenta, very thick]--(12);
\draw(12)[magenta, very thick]--(10);

\node[draw,align=left] at (.7,-1.4) {$D_5$};
\end{tikzpicture}

\vspace{.7cm}

\begin{tikzpicture}[node distance = .6cm]
\node(13)[circle, draw, fill=pink] at (.7, 0) {};
\node(3)[circle, draw, fill=pink] at (2, 0) {};
\node(4)[circle, draw, fill=pink] at (1.35, 0) {};
\node(5)[circle, draw, fill=pink] at (1, -.5) {};
\node(6)[circle, draw, fill=pink] at (1.7, -.5) {};
\node(9)[circle, draw, fill=pink] at (1.35, -1) {};

\draw(3)[magenta, very thick]--(4);
\draw(3)[magenta, very thick]--(6);
\draw(4)[magenta, very thick]--(5);
\draw(6)[magenta, very thick]--(9);
\draw(5)[magenta, very thick]--(6);
\draw(6)[magenta, very thick]--(4);
\draw(9)[magenta, very thick]--(5);
\draw(13)[magenta, very thick]--(4);
\draw(5)[magenta, very thick]--(13);

\node[draw,align=left] at (1.4,-1.7) {$D_7$};

\end{tikzpicture}
\hspace{.5cm}
\begin{tikzpicture}[node distance = .6cm]
\node(13) at (.5,0) {};
\node(3)[circle, draw, fill=pink][below of = 13]{};
\node(4)[circle, draw, fill=pink][right of = 13]{};
\node(5)[circle, draw, fill=pink][below of = 4]{};
\node(6)[circle, draw, fill=pink][right of = 5]{};
\node(9)[circle, draw, fill=pink][above of = 6]{};
\node(12)[circle, draw, fill=pink][right of = 6]{};

\draw(3)[magenta, very thick]--(4);
\draw(3)[magenta, very thick]--(5);
\draw(4)[magenta, very thick]--(5);
\draw(4)[magenta, very thick]--(9);
\draw(5)[magenta, very thick]--(6);
\draw(6)[magenta, very thick]--(4);
\draw(9)[magenta, very thick]--(6);
\draw(12)[magenta, very thick]--(6);
\draw(12)[magenta, very thick]--(9);

\node[draw,align=left] at (1.4,-1.2) {$D_8$};

\end{tikzpicture}
\hspace{.5cm}
\begin{tikzpicture}[node distance = .6cm]
\node(13)[circle, draw, fill=pink] at (-.5, 0) {};
\node(3)[circle, draw, fill=pink][below of = 13]{};
\node(4)[circle, draw, fill=pink][right of = 13]{};
\node(5)[circle, draw, fill=pink][below of = 4]{};
\node(6)[circle, draw, fill=pink][right of = 5]{};
\node(9)[circle, draw, fill=pink][above of = 6]{};

\draw(3)[magenta, very thick]--(4);
\draw(3)[magenta, very thick]--(5);
\draw(4)[magenta, very thick]--(5);
\draw(4)[magenta, very thick]--(9);
\draw(5)[magenta, very thick]--(6);
\draw(6)[magenta, very thick]--(4);
\draw(9)[magenta, very thick]--(5);
\draw(13)[magenta, very thick]--(4);
\draw(3)[magenta, very thick]--(13);

\node[draw,align=left] at (0,-1.2) {$D_9$};

\end{tikzpicture}

\vspace{.7cm}
\begin{tikzpicture}[node distance = .6cm]
\node(1)[circle, draw, fill=pink] at (1, .2) {};
\node(2)[circle, draw, fill=pink] at (1.8, .2) {};
\node(5)[circle, draw, fill=pink] at (1.4, -.2) {};
\node(4)[circle, draw, fill=pink] at (1, -.6) {};
\node(3)[circle, draw, fill=pink] at (1.8, -.6) {};

\draw(1)[magenta, very thick]--(2);
\draw(2)[magenta, very thick]--(3);
\draw(3)[magenta, very thick]--(4);
\draw(4)[magenta, very thick]--(1);
\draw(1)[magenta, very thick]--(5);
\draw(2)[magenta, very thick]--(5);
\draw(3)[magenta, very thick]--(5);
\draw(4)[magenta, very thick]--(5);

\node[draw,align=left] at (1.4,-1.2) {$D_{11}$};

\end{tikzpicture}
\hspace{.9cm}
\begin{tikzpicture}[node distance = .6cm]
\node(13)[circle, draw, fill=pink] at (.7, 0) {};
\node(3)[circle, draw, fill=pink][below of = 13]{};
\node(4)[circle, draw, fill=pink][right of = 13]{};
\node(5)[circle, draw, fill=pink][below of = 4]{};
\node(6)[circle, draw, fill=pink][right of = 5]{};
\node(9)[circle, draw, fill=pink][above of = 6]{};

\draw(3)[magenta, very thick]--(4);
\draw(3)[magenta, very thick]--(5);
\draw(4)[magenta, very thick]--(5);
\draw(4)[magenta, very thick]--(9);
\draw(5)[magenta, very thick]--(6);
\draw(6)[magenta, very thick]--(4);
\draw(9)[magenta, very thick]--(5);
\draw(13)[magenta, very thick]--(4);
\draw(5)[magenta, very thick]--(13);

\node[draw,align=left] at (1.4,-1.2) {$D_{12}$};

\end{tikzpicture}
\hspace{.9cm}
\begin{tikzpicture}[node distance = .6cm]
\node(13)[circle, draw, fill=pink] at (1.0, 0) {};
\node(3)[circle, draw, fill=pink][below of = 13]{};
\node(4)[circle, draw, fill=pink][right of = 13]{};
\node(5)[circle, draw, fill=pink][below of = 4]{};

\draw(3)[magenta, very thick]--(4);
\draw(3)[magenta, very thick]--(5);
\draw(3)[magenta, very thick]--(13);
\draw(4)[magenta, very thick]--(5);
\draw(5)[magenta, very thick]--(13);
\draw(13)[magenta, very thick]--(4);
\draw(5)[magenta, very thick]--(13);

\node[draw,align=left] at (1.4,-1.2) {$D_{13}$};

\end{tikzpicture}
\end{center}

Using Gr\"unbaum's result, we note that every 4-colorable planar graph must have at least four 3-cycles.  By the restrictions we placed on the graphs of $\mathbb{F}_4$, we know that for any 4-colorable planar graph $G$, there must be an $G' \subset G$ and $H \in \mathbb{F}_4$ such that either $G' \cong H$, or $G'$ is an edge-minimal subgraph of $G$ with four triangles such that $\mathbb{E}[\chi(G'_p)]\geq \mathbb{E}[\chi H_p)]$ for all $p$ (this would be the case where we can separate at least one vertex in $G'$).  Therefore, we must only consider the graphs in $\mathbb{F}_4$ along with the expected chromatic numbers of their random subgraphs to determine whether $K_4$ is a 4-minimizer.  Note that in the following, we have renamed the graphs of $\mathbb{F}_4$ for convenience:

\begin{center}
\begin{tikzpicture}[node distance = .6cm]
\node(1)[circle, draw, fill=pink] at (-.2,0) {};
\node(2)[circle, draw, fill=pink][below of = 1]{};
\node(3)[circle, draw, fill=pink][right of = 2]{};
\node(13)[right of = 1]{};
\node(4)[circle, draw, fill=pink][right of = 13]{};
\node(5)[circle, draw, fill=pink][below of = 4]{};
\node(6)[circle, draw, fill=pink][right of = 5]{};
\node(14)[right of = 4]{};
\node(7)[circle, draw, fill=pink][right of = 14]{};
\node(8)[circle, draw, fill=pink][below of = 7]{};
\node(9)[circle, draw, fill=pink][right of = 8]{};
\node(15)[right of = 7]{};
\node(10)[circle, draw, fill=pink][right of = 15]{};
\node(11)[circle, draw, fill=pink][below of = 10]{};
\node(12)[circle, draw, fill=pink][right of = 11]{};

\draw(1)[magenta, very thick]--(2);
\draw(2)[magenta, very thick]--(3);
\draw(3)[magenta, very thick]--(1);
\draw(4)[magenta, very thick]--(5);
\draw(5)[magenta, very thick]--(6);
\draw(6)[magenta, very thick]--(4);
\draw(7)[magenta, very thick]--(8);
\draw(8)[magenta, very thick]--(9);
\draw(9)[magenta, very thick]--(7);
\draw(10)[magenta, very thick]--(11);
\draw(11)[magenta, very thick]--(12);
\draw(12)[magenta, very thick]--(10);

\node[draw,align=left] at (1.8,-1.2) {$\mathbb{E}[\chi(G^1_{1/2})] \approx 2.4136$};
\end{tikzpicture}
\hspace{.6cm}
\begin{tikzpicture}[node distance = .6cm]
\node(13){};
\node(3)[circle, draw, fill=pink][below of = 13]{};
\node(4)[circle, draw, fill=pink][right of = 13]{};
\node(5)[circle, draw, fill=pink][below of = 4]{};
\node(6)[circle, draw, fill=pink][right of = 5]{};
\node(14)[right of = 4]{};
\node(7)[circle, draw, fill=pink][right of = 14]{};
\node(8)[circle, draw, fill=pink][below of = 7]{};
\node(9)[circle, draw, fill=pink][right of = 8]{};
\node(15)[right of = 7]{};
\node(10)[circle, draw, fill=pink][right of = 15]{};
\node(11)[circle, draw, fill=pink][below of = 10]{};
\node(12)[circle, draw, fill=pink][right of = 11]{};

\draw(3)[magenta, very thick]--(4);
\draw(3)[magenta, very thick]--(5);
\draw(4)[magenta, very thick]--(5);
\draw(5)[magenta, very thick]--(6);
\draw(6)[magenta, very thick]--(4);
\draw(7)[magenta, very thick]--(8);
\draw(8)[magenta, very thick]--(9);
\draw(9)[magenta, very thick]--(7);
\draw(10)[magenta, very thick]--(11);
\draw(11)[magenta, very thick]--(12);
\draw(12)[magenta, very thick]--(10);

\node[draw,align=left] at (1.8,-1.2) {$\mathbb{E}[\chi(G^2_{1/2})] \approx 2.4014$};

\end{tikzpicture}
\hspace{.6cm}
\begin{tikzpicture}[node distance = .6cm]
\node(13){};
\node(3)[circle, draw, fill=pink][below of = 13]{};
\node(4)[circle, draw, fill=pink][right of = 13]{};
\node(5)[circle, draw, fill=pink][below of = 4]{};
\node(6)[circle, draw, fill=pink][right of = 5]{};
\node(14)[right of = 4]{};
\node(9)[circle, draw, fill=pink][right of = 6]{};
\node(10)[circle, draw, fill=pink][right of = 9]{};
\node(11)[circle, draw, fill=pink][above of = 10]{};
\node(12)[circle, draw, fill=pink][right of = 10]{};

\draw(3)[magenta, very thick]--(4);
\draw(3)[magenta, very thick]--(5);
\draw(4)[magenta, very thick]--(5);
\draw(5)[magenta, very thick]--(6);
\draw(6)[magenta, very thick]--(4);
\draw(9)[magenta, very thick]--(10);
\draw(9)[magenta, very thick]--(11);
\draw(10)[magenta, very thick]--(11);
\draw(11)[magenta, very thick]--(12);
\draw(12)[magenta, very thick]--(10);

\node[draw,align=left] at (1.4,-1.2) {$\mathbb{E}[\chi(G^3_{1/2})] \approx 2.3887$};

\end{tikzpicture}

\vspace*{1cm}

\begin{tikzpicture}[node distance = .6cm]
\node(13){};
\node(3)[circle, draw, fill=pink][below of = 13]{};
\node(4)[circle, draw, fill=pink][right of = 13]{};
\node(5)[circle, draw, fill=pink][below of = 4]{};
\node(6)[circle, draw, fill=pink][right of = 5]{};
\node(9)[circle, draw, fill=pink][above of = 6]{};
\node(10)[circle, draw, fill=pink][right of = 6]{};
\node(11)[circle, draw, fill=pink][above of = 10]{};
\node(12)[circle, draw, fill=pink][right of = 10]{};

\draw(3)[magenta, very thick]--(4);
\draw(3)[magenta, very thick]--(5);
\draw(4)[magenta, very thick]--(5);
\draw(4)[magenta, very thick]--(9);
\draw(5)[magenta, very thick]--(6);
\draw(6)[magenta, very thick]--(4);
\draw(9)[magenta, very thick]--(6);
\draw(10)[magenta, very thick]--(11);
\draw(11)[magenta, very thick]--(12);
\draw(12)[magenta, very thick]--(10);

\node[draw,align=left] at (1.4,-1.2) {$\mathbb{E}[\chi(G^4_{1/2})] \approx 2.3975$};
\end{tikzpicture}
\hspace{.9cm}
\begin{tikzpicture}[node distance = .6cm]
\node(13) at (.3,0) {};
\node(3)[circle, draw, fill=pink][below of = 13]{};
\node(4)[circle, draw, fill=pink][right of = 13]{};
\node(5)[circle, draw, fill=pink][below of = 4]{};
\node(6)[circle, draw, fill=pink][right of = 5]{};
\node(9)[circle, draw, fill=pink][above of = 6]{};
\node(10)[circle, draw, fill=pink][right of = 6]{};
\node(11)[circle, draw, fill=pink][above of = 10]{};
\node(12)[circle, draw, fill=pink][right of = 10]{};

\draw(3)[magenta, very thick]--(4);
\draw(3)[magenta, very thick]--(5);
\draw(4)[magenta, very thick]--(5);
\draw(4)[magenta, very thick]--(9);
\draw(5)[magenta, very thick]--(6);
\draw(6)[magenta, very thick]--(4);
\draw(9)[magenta, very thick]--(5);
\draw(10)[magenta, very thick]--(11);
\draw(11)[magenta, very thick]--(12);
\draw(12)[magenta, very thick]--(10);

\node[draw,align=left] at (1.4,-1.2) {$\mathbb{E}[\chi(G^5_{1/2})] \approx 2.3770$};

\end{tikzpicture}
\hspace{.9cm}
\begin{tikzpicture}[node distance = .6cm]
\node(13) at (.5,0) {};
\node(3)[circle, draw, fill=pink][below of = 13]{};
\node(4)[circle, draw, fill=pink][right of = 13]{};
\node(5)[circle, draw, fill=pink][below of = 4]{};
\node(6)[circle, draw, fill=pink][right of = 5]{};
\node(9)[circle, draw, fill=pink][above of = 6]{};
\node(12)[circle, draw, fill=pink][right of = 6]{};

\draw(3)[magenta, very thick]--(4);
\draw(3)[magenta, very thick]--(5);
\draw(4)[magenta, very thick]--(5);
\draw(4)[magenta, very thick]--(9);
\draw(5)[magenta, very thick]--(6);
\draw(6)[magenta, very thick]--(4);
\draw(9)[magenta, very thick]--(6);
\draw(12)[magenta, very thick]--(6);
\draw(12)[magenta, very thick]--(9);

\node[draw,align=left] at (1.4,-1.2) {$\mathbb{E}[\chi(G^6_{1/2})] \approx 2.3906$};

\end{tikzpicture}

\vspace{.5cm}
\begin{tikzpicture}[node distance = .6cm]
\node(13)[circle, draw, fill=pink] at (-.5, 0) {};
\node(3)[circle, draw, fill=pink][below of = 13]{};
\node(4)[circle, draw, fill=pink][right of = 13]{};
\node(5)[circle, draw, fill=pink][below of = 4]{};
\node(6)[circle, draw, fill=pink][right of = 5]{};
\node(9)[circle, draw, fill=pink][above of = 6]{};

\draw(3)[magenta, very thick]--(4);
\draw(3)[magenta, very thick]--(5);
\draw(4)[magenta, very thick]--(5);
\draw(4)[magenta, very thick]--(9);
\draw(5)[magenta, very thick]--(6);
\draw(6)[magenta, very thick]--(4);
\draw(9)[magenta, very thick]--(5);
\draw(13)[magenta, very thick]--(4);
\draw(3)[magenta, very thick]--(13);

\node[draw,align=left] at (0,-1.2) {$\mathbb{E}[\chi(G^7_{1/2})] \approx 2.3809$};

\end{tikzpicture}
\hspace{.9cm}
\begin{tikzpicture}[node distance = .6cm]
\node(13)[circle, draw, fill=pink] at (.7, 0) {};
\node(3)[circle, draw, fill=pink][below of = 13]{};
\node(4)[circle, draw, fill=pink][right of = 13]{};
\node(5)[circle, draw, fill=pink][below of = 4]{};
\node(6)[circle, draw, fill=pink][right of = 5]{};
\node(9)[circle, draw, fill=pink][above of = 6]{};

\draw(3)[magenta, very thick]--(4);
\draw(3)[magenta, very thick]--(5);
\draw(4)[magenta, very thick]--(5);
\draw(4)[magenta, very thick]--(9);
\draw(5)[magenta, very thick]--(6);
\draw(6)[magenta, very thick]--(4);
\draw(9)[magenta, very thick]--(5);
\draw(13)[magenta, very thick]--(4);
\draw(5)[magenta, very thick]--(13);

\node[draw,align=left] at (1.4,-1.2) {$\mathbb{E}[\chi(G^8_{1/2})] \approx 2.3398$};

\end{tikzpicture}
\hspace{.9cm}
\begin{tikzpicture}[node distance = .6cm]
\node(1)[circle, draw, fill=pink] at (1, .2) {};
\node(2)[circle, draw, fill=pink] at (1.8, .2) {};
\node(5)[circle, draw, fill=pink] at (1.4, -.2) {};
\node(4)[circle, draw, fill=pink] at (1, -.6) {};
\node(3)[circle, draw, fill=pink] at (1.8, -.6) {};

\draw(1)[magenta, very thick]--(2);
\draw(2)[magenta, very thick]--(3);
\draw(3)[magenta, very thick]--(4);
\draw(4)[magenta, very thick]--(1);
\draw(1)[magenta, very thick]--(5);
\draw(2)[magenta, very thick]--(5);
\draw(3)[magenta, very thick]--(5);
\draw(4)[magenta, very thick]--(5);

\node[draw,align=left] at (1.4,-1.2) {$\mathbb{E}[\chi(G^9_{1/2})] \approx 2.3828$};

\end{tikzpicture}

\vspace{.5cm}
\begin{tikzpicture}[node distance = .6cm]
\node(13)[circle, draw, fill=pink] at (.7, 0) {};
\node(3)[circle, draw, fill=pink] at (2, 0) {};
\node(4)[circle, draw, fill=pink] at (1.35, 0) {};
\node(5)[circle, draw, fill=pink] at (1, -.5) {};
\node(6)[circle, draw, fill=pink] at (1.7, -.5) {};
\node(9)[circle, draw, fill=pink] at (1.35, -1) {};

\draw(3)[magenta, very thick]--(4);
\draw(3)[magenta, very thick]--(6);
\draw(4)[magenta, very thick]--(5);
\draw(6)[magenta, very thick]--(9);
\draw(5)[magenta, very thick]--(6);
\draw(6)[magenta, very thick]--(4);
\draw(9)[magenta, very thick]--(5);
\draw(13)[magenta, very thick]--(4);
\draw(5)[magenta, very thick]--(13);

\node[draw,align=left] at (1.4,-1.7) {$\mathbb{E}[\chi(G^{10}_{1/2})] \approx 2.3984$};

\end{tikzpicture}
\hspace{.9cm}
\begin{tikzpicture}[node distance = .6cm]
\node(13)[circle, draw, fill=pink] at (1.0, 0) {};
\node(3)[circle, draw, fill=pink][below of = 13]{};
\node(4)[circle, draw, fill=pink][right of = 13]{};
\node(5)[circle, draw, fill=pink][below of = 4]{};

\draw(3)[magenta, very thick]--(4);
\draw(3)[magenta, very thick]--(5);
\draw(3)[magenta, very thick]--(13);
\draw(4)[magenta, very thick]--(5);
\draw(5)[magenta, very thick]--(13);
\draw(13)[magenta, very thick]--(4);
\draw(5)[magenta, very thick]--(13);

\node[draw,align=left] at (1.4,-1.2) {$\mathbb{E}[\chi((K_4)_{1/2})] \approx 2.3594$};

\end{tikzpicture}
\end{center}

From the expected values of the different configurations, we see that if $G$ is a planar 4-minimizer, then either it must be $K_4$, or it must be a 4-critical supergraph of $G^8$.  Let us consider the graph of $G^8$ a little more carefully (this time drawn in a planar fashion):

\begin{center}
\begin{tikzpicture}[node distance = 2cm]
\node(14){};
\node(13)[circle, draw, fill=pink][below of = 14]{C};
\node(15)[below of = 13]{};
\node(16)[right of = 14]{};
\node(3)[circle, draw, fill=pink][right of = 13]{D};
\node(4)[circle, draw, fill=pink][right of = 16]{A};
\node(17)[below of = 4]{};
\node(5)[circle, draw, fill=pink][below of = 17]{B};
\node(17)[right of = 5]{};
\node(6)[circle, draw, fill=pink][above of = 17]{E};
\node(9)[circle, draw, fill=pink][right of = 6]{F};

\draw(3)[magenta, very thick]--(4);
\draw(3)[magenta, very thick]--(5);
\draw(4)[magenta, very thick]--(5);
\draw(4)[magenta, very thick]--(9);
\draw(5)[magenta, very thick]--(6);
\draw(6)[magenta, very thick]--(4);
\draw(9)[magenta, very thick]--(5);
\draw(13)[magenta, very thick]--(4);
\draw(5)[magenta, very thick]--(13);

\end{tikzpicture}
\end{center}

No matter how we draw $G^8$, we will always have a subgraph structure of a 3-cycle enclosing a fourth point that connects to two of the points of the 3-cycle.  Without loss of generality then, let us suppose that the 3-cycle $C_3$ is composed of points $A$, $B$, and $C$, with interior point $D$:

\begin{center}
\begin{tikzpicture}[node distance = 1cm]
\node(14){};
\node(13)[circle, draw, fill=yellow][below of = 14]{C};
\node(15)[below of = 13]{};
\node(16)[right of = 14]{};
\node(3)[circle, draw, fill=yellow][right of = 13]{D};
\node(4)[circle, draw, fill=yellow][right of = 16]{A};
\node(17)[below of = 4]{};
\node(5)[circle, draw, fill=yellow][below of = 17]{B};
\node(17)[right of = 5]{};
\node(6)[circle, draw, fill=pink][above of = 17]{E};
\node(9)[circle, draw, fill=pink][right of = 6]{F};

\draw(3)[yellow, very thick]--(4);
\draw(3)[yellow, very thick]--(5);
\draw(4)[yellow, very thick]--(5);
\draw(4)[magenta, very thick]--(9);
\draw(5)[magenta, very thick]--(6);
\draw(6)[magenta, very thick]--(4);
\draw(9)[magenta, very thick]--(5);
\draw(13)[yellow, very thick]--(4);
\draw(5)[yellow, very thick]--(13);

\end{tikzpicture}
\end{center}

Now treat $G^8$ as the subgraph of a larger graph $H$.  Consider Int($C_3$) (the subgraph induced by the vertices on and inside $C_3$) and Ext($C_3$) (the subgraph induced by the vertices on and outside $C_3$).  Because the coloring of $C_3$ is independent of the structure of $H$ (up to translation), the coloring of Int($C_3$) is independent of the coloring of Ext($C_3$).  Therefore, if $H$ is 4-colorable, then at least one of Int($C_3$) or Ext($C_3$) must be 4-colorable.  However, because Int($C_3$)$\setminus C_3$ and Ext($C_3$)$\setminus C_3$ are both nonempty, $H$ cannot be critically chromatic.  Therefore, there exists no planar 4-critical chromatic graph containing $G^8$.  Therefore, $K_4$ is the unique planar 4-minimizer for $p=\frac{1}{2}$.

In general, the proof that $K_4$ is the unique planar 4-minimizer for all $p\in(0,\frac{1}{2}]$ follows from the same reasoning.  Namely, we simply find polynomial expressions in $p$ for the expected values of each our different triangle configurations that must appear---each such polynomial is relatively easy to calculate, but they are emphatically awful to look at.  Among these, the polynomial for $K_4$ is the least for all $p \in (0, 1/2]$, which follows from routine computations.
\end{proof}
\end{appendices}
\end{document}